\documentclass[preprint,10pt]{elsarticle}

\usepackage{amssymb}
 \usepackage{amsthm}
 \usepackage{amsmath}
 \usepackage{amsfonts}

\usepackage[margin=1in]{geometry}

\newtheorem{thm}{Theorem}[section]
\newtheorem{lem}[thm]{Lemma}

\newcommand{\R}{\mathbb{R}}

\newcommand{\LL}{\mathcal{L}}
\newcommand{\E}{\mathcal{E}}
\newcommand{\Q}{\mathcal{Q}}

\journal{Applied Mathematics and Optimization}

\begin{document}

\begin{frontmatter}



\title{On the Stability for the Cauchy Problem of Timoshenko thermoelastic Systems with Past History: Cattaneo and Fourier Law}

\author{Fernando A. Gallego}
\ead{fagallegor@unal.edu.co}
\address{Departamento de Matem\'aticas, Universidad Nacional de Colombia - Sede Manizales, (UNAL), Cra 27 No. 64-60, 170003, Manizales, Colombia.}

\author{Hugo D. Fern\'andez Sare}
\address{Departamento de Matem\'atica, Universidade Federal de Juiz de Fora, CEP 36036-900, Juiz de Fora, MG, Brazil.}
\ead{hugo.sare@ice.ufjf.br}

\begin{abstract}
In this paper, we investigate the decay properties of thermoelastic Timoshenko systems with past history in the whole space where the thermal effects are modeled by Cattaneo and Fourier laws. We establish rates of decay of order $(1+t)^{-\frac{1}{8}}$ for both systems, Timoshenko-Fourier and Timoshenko-Cattaneo, satisfying the regularity-loss type property. Moreover, for the Cattaneo case, we show that the decay rates depend on a new condition $\chi_{0,\tau}$ which has been recently introduced to study the asymptotic behavior of Timoshenko systems in bounded domains. We found that this number also plays an important role in unbounded situations, affecting the decay rate of the solution.
\end{abstract}




\end{frontmatter}



\section{Introduction}

In the literature concerning Timoshenko systems, the stability nature  of solutions have a relationship with the wave speeds of propagation, essentially, when these speeds are equal or different. In this context, denoting by $\chi_0$ the difference of propagation's speed (see equation \eqref{controlnumber'}), we investigate how the decay rates of solutions of thermoelastic Timoshenko systems, depend of $\chi_0$, in particular when the thermal effects are given by Cattaneo and Fourier law, both with additional history terms. As we can see from the references, the constant $\chi_0$ will play an important role in the characterization of asymptotic behavior of solutions.

Recalling that Cattaneo's law is a hyperbolic heat model implying that the temperature has a finite speed of propagation, we can observe the impact of this heat model on the stability of Timoshenko systems. Being more specific, recently in \cite{fatori2014,santos2012}, the authors proved that Cattaneo's law modifies the stability number $\chi_0$ when the model is formulated in bounded domains. Since this hyperbolic model generates dissipative thermal effects weaker than the parabolic Fourier model, the authors introduce a new stability number $\chi_{0,\tau}$ (see equation \eqref{controlnumber}), which generalizes the previous one $\chi_0$ in the sense that, when $\tau = 0$, Cattaneo's law turns into the Fourier law and the conditions over the new number $\chi_{0,\tau}$ are
equivalent to the old stability number $\chi_0$. 

In this line of research, our goal in this paper is to investigate the relation between $\chi_{0,\tau}$ and $\chi_{0}$ with the decay rates of the solution of Timoshenko system posed in the whole real line. In fact, we consider the Cauchy problem of the Timoshenko system with the heat conduction described by the Cattaneo and Fourier law with a history term, given by
\begin{equation}\label{system0}
\left\lbrace\begin{tabular}{l l}
$\rho_1\varphi _{tt}-k \left ( \varphi_x - \psi \right )_x = 0$ & in $ (0,\infty) \times \R$, \\
$\rho_2\psi_{tt}-b\psi_{xx} +m\displaystyle \int_0^{\infty}g(s)\psi_{xx}(t-s,x)ds - k\left ( \varphi_x - \psi\right )+\delta \theta_{x} =0$ & in $ (0,\infty) \times \R,$ \\
$\rho_3\theta_{t} +q_x +\delta \psi_{xt}=0$ & in $ (0,\infty) \times \R,$\\
$\tau q_t+\beta q +\theta_x=0$ & in $ (0,\infty) \times \R.$
\end{tabular} \right.
\end{equation}
where $b, k, m, \delta, \beta, \rho_1, \rho_2, \rho_3, $ and $\tau$ are positive constants; $\psi(t,x)$ has been assigned on $(-\infty,0] \times \R$, with initial data
\begin{equation}\label{system01}
\begin{tabular}{l l l }
$\varphi(0,\cdot)=\varphi_0(\cdot)$,   & $\psi(s,\cdot)=\psi_0(s,\cdot)$,   & $\theta(0,\cdot)=\theta_0(\cdot),\quad \forall s\in(-\infty,0]$\\
$\varphi_t(0,\cdot)=\varphi_1(\cdot)$, & $\psi_t(0,\cdot)=\psi_1(\cdot)$, & $q(0,\cdot)=q_0(\cdot)$,
\end{tabular}
\end{equation}
with $\phi, \psi, \theta$ and $q$  denoting the transversal displacement, the rotation angle of the beam, the temperature and the heat flow, respectively.
The integral term represents a history term with kernel $g$ satisfying the following hypotheses:
\begin{enumerate}
\item[$(H_1)$] $g(\cdot)$ is a non negative function.
\item[$(H_2)$] There exist positive constants $k_1$ and $k_2$, such that, $-k_1 g(s) \leq g'(s) \leq -k_2 g(s).$
\item[$(H_3)$] $a:=b-b_0>0$, where $b_0=\int_0^{\infty}g(s)ds$.
\end{enumerate}
For this system, in bounded domains $(0,L)$, see for instance \cite{fatori2014}; the associated stability number is given by
\begin{equation}\label{controlnumber}
\chi_{0,\tau}=\left( \tau-\frac{\rho_1}{\rho_3 k}\right)\left( \rho_2 -\frac{b \rho_1}{k}\right) - \frac{\tau \rho_1 \delta^2}{\rho_3 k}.
\end{equation}
As mentioned in several references, condition \eqref{controlnumber} is more mathematical than physical, because is not realistic to assume that the propagation speeds associated to system \eqref{system0} will satisfy condition \eqref{controlnumber}. Note that, when $\tau=0$, the system \eqref{system0} has a thermal effect given by the Fourier law ($q= -\widetilde{\beta}\theta_x$). Indeed, formally the Timoshenko-Cattaneo system \eqref{system0} is reduced to the following Timoshenko-Fourier system
\begin{equation}\label{system0'}
\left\lbrace\begin{tabular}{l l}
$\rho_1\varphi _{tt}-k \left ( \varphi_x - \psi \right )_x = 0$ & in $ (0,\infty) \times \R$, \\
$\rho_2\psi_{tt}-b\psi_{xx} +m\displaystyle \int_0^{\infty}g(s)\psi_{xx}(t-s,x)ds - k\left ( \varphi_x - \psi\right )+\delta \theta_{x} =0$ & in $ (0,\infty) \times \R,$ \\
$\rho_3\theta_{t} -\widetilde{\beta}\theta_{xx} +\delta \psi_{xt}=0$ & in $ (0,\infty) \times \R,$
\end{tabular} \right.
\end{equation}
with initial data
\begin{equation}\label{system01'}
\begin{tabular}{l l l }
$\varphi(0,\cdot)=\varphi_0(\cdot)$,   & $\psi(s,\cdot)=\psi_0(s,\cdot)$,   & $\theta(0,\cdot)=\theta_0(\cdot),\quad \forall s\in(-\infty,0]$,\\
$\varphi_t(0,\cdot)=\varphi_1(\cdot)$, & $\psi_t(0,\cdot)=\psi_1(\cdot)$, &
\end{tabular}
\end{equation}
and stability number given by
\begin{equation}\label{controlnumber'}
\chi_{0}= \rho_2 -\frac{b \rho_1}{k}.
\end{equation}
As we said previously, the main  purpose of this article is to investigate  the relationship between damping terms, the stability numbers $\chi_{0,\tau}, \chi_0$ and their influence on the decay rate of solutions of systems \eqref{system0}-\eqref{system01} and \eqref{system0'}-\eqref{system01'}, respectively.

Let us start by giving some references on Timoshenko systems. The original Timoshenko system was first introduced by Timoshenko \cite{timoshenko1921,timoshenko1922} and describes the vibration of a beam taking into account the transversal displacement and the rotational angle of the beam filaments. An initial boundary value problem associated to \eqref{system0} and \eqref{system0'} under hypotheses  $(H_1), (H_2), (H_3)$ was considered by Fern\'andez Sare and Racke in \cite{sare2009}. They prove that the energy of the solution for the Timoshenko-Cattaneo model with history does not decay exponentially as $t\rightarrow \infty$ if $\chi_0=0$, while for the Timoshenko-Fourier system the energy decays exponentially if and only if $\chi_0=0$.
This result has been recently improved by Fatori et al in \cite{fatori2014} where, for the Cattaneo's case, the exponential stability is obtained if and only if a new condition on the wave speeds of propagation is satisfied, i.e, the energy of solution of a IBVP Timoshenko-Cattaneo decay exponentially if and only if $\chi_{0,\tau}=0$, where $\chi_{0,\tau}$ is given by \eqref{controlnumber}. Furthermore, if $\chi_{0,\tau}\neq 0$, they prove that the energy decays polynomially with rate $t^{-\frac{1}{2}}$.

There are many other references on Timoshenko systems in  bounded domains with interesting results. In particular, the problem of stability for Timoshenko-type systems in bounded domains has received much attention in the last years, and quite a number of results concerning uniform and asymptotic decay of energy have been established, see for instance \cite{alabau2007, rivera2003, messaoudi2009, guesmia,  keddi2018,   mustafa2009, messaoudi2004,  said2009, rivera20031, rivera2008, sare2008, park2011, soufyane2003} and references therein. As a matter of fact, in bounded domains the proofs of stability results for Timoshenko systems are based on  Poincar\'e inequalities and boundary conditions of the systems.

In this paper we are specially interested in the unbounded situation: when the system is formulated in the whole space $\R$. This kind of problem has been considered in recent papers because it exhibits the \textit{regularity-loss phenomenon} that usually appears in the pure Cauchy problems; see for instance \cite{hosono2006,ide2008,racke2010,ueda2011} and references therein. Roughly speaking, a decay rate of solution is of regularity-loss type when it is obtained only by assuming some additional regularity on the initial conditions. In this direction, we can mentioned some recent results on stabilization of Cauchy Timoshenko systems. For instance, in Ide-Haramoto-Kawashima \cite{ide2008}, Ide-Kawashima \cite{kawashima2008} and Racke-Houari \cite{houari2013,racke2010}, the authors consider Timoshenko systems with normalized coefficients proving that the assumptions $b=1$ or $b\neq1$ play decisive roles in showing whether or not the decay estimates of solutions are of regularity-loss type.

For Cauchy problems associated to Timoshenko systems in thermoelasticity, as far as we know, the decay rate of solutions has been first
studied by Said-Houari and Kasimov in \cite{houari2012,said2013damping}. In particular, the authors proved in  \cite{said2013damping} that the Timoshenko system couppled with Cattaneo or Fourier law have the same rate of decay, this is, the solutions $W= (\varphi_t,\psi_t,a\psi_x,\varphi_x-\psi,\theta)^T$ decay with the rate:
\begin{equation*}
\|\partial_x^k W(t)\|_{L^2} \leq C(1+t)^{-\frac{1}{12}-\frac{k}{6}}\|W_0\|_{L^1}+Ce ^{-ct}\|\partial_x^{k+l}W_0\|_{L^2}
\end{equation*}
for $a=1$ and
\begin{equation*}
\|\partial_x^k W(t)\|_{L^2} \leq C(1+t)^{-\frac{1}{12}-\frac{k}{6}}\|W_0\|_{L^1}+C(1+t)^{-\frac{l}{2}}\|\partial_x^{k+l}W_0\|_{L^2}
\end{equation*}
for $a\neq 1$, $k=1,2,...,s-l$. In \cite{houari2012}, considering an additional frictional damping $\lambda \psi_t(x,t)$ in the second equation, they obtain the same decay estimates with optimal rates $(1 + t)^{-\frac{1}{4}-\frac{k}{2}}$. More recently, Khader and Said-Houari in \cite{khadersaid2016} studied the Cauchy problem for the Timoshenko system with the  Gurtin-Pipkin thermal law:
\begin{equation}\label{gurtinpinkin}
\left\lbrace\begin{tabular}{l l}
$\varphi _{tt}-\left ( \varphi_x - \psi \right )_x = 0$ & in $ (0,\infty) \times \R$, \\
$\psi_{tt}-a^2\psi_{xx}  - \left ( \varphi_x - \psi\right )+\delta \theta_{x} =0$ & in $ (0,\infty) \times \R,$ \\
$\theta_t-\frac{1}{\beta}\displaystyle \int_0^{\infty}g(s)\theta_{xx}(t-s,x)ds+ \delta \psi_{tx}=0$ & in $ (0,\infty) \times \R,$
\end{tabular} \right.
\end{equation}
where the memory kernel $g(s)$ is a convex summable function on $[0,\infty)$ with total mass equal to 1.  They proved that the rate of decay depends of the number $\alpha_g:=\left((g(0))^{-1}\beta -1\right)\left(1-a^2\right)-(g(0))^{-1}\delta^2\beta$, which also controls the rate in bounded domains, see \cite{pata2014}. Additionally, we can cite some recent papers studying more general beam models posed in the whole real line, \cite{gallego2017, khadersaid2017, houari2016, houari2016-2, houari2015, houari2014}.

The main goal of this paper is to investigate the decay rate of the Cauchy problems \eqref{system0} and \eqref{system0'}. We prove that the same number $\chi_{0,\tau}$ defined in \eqref{controlnumber}, which controls the behavior of the solution in bounded domains \cite{fatori2014}, also plays an important role in unbounded domains and affects the decay rate of solutions, see Theorems \ref{teocattaneo} and \ref{teofourier} below. More precisely, we show that the respective solutions of Timoshenko-Cattaneo and Timoshenko-Fourier with history term, are of regularity-loss type and decay slowly with the rate $(1+ t)^{-\frac{1}{8}}$  in the $L^2$-norm.  Our proofs are based on some estimates for the Fourier image of the solution, Plancherel Theorem,  as well as on a suitable linear combination of series associated to energy estimates. Here, the decay rate $(1 + t)^{-\frac{1}{8}}$ will be obtained by taking regular initial data $U_0 \in  H^s(\R)$, for some $s\in\R$. This regularity loss comes from the analysis of the Fourier image, $\hat{U}(\xi, t)$, of the solution $U(\xi, t)$. In fact, we will obtain the estimate 
\begin{equation*}
\left | \hat{U}(\xi, t)\right|^2 \leq C e^{-\beta \rho(\xi)t}\left | \hat{U}(\xi, 0)\right|^2,
\end{equation*}
where $C, \beta$ are positive constants and
\begin{equation*}
\rho(\xi) = \begin{cases}
\dfrac{\xi^4}{\left(1+\xi^2\right)^3}, & \text{if $\chi_{0,\tau} = 0$ (resp, $\chi_{0} = 0$)}, \\
\\
\dfrac{\xi^4}{\left(1+\xi^2\right)^4}, &\text{if  $\chi_{0,\tau} \neq 0$ (resp, $\chi_{0} \neq 0$)}.
\end{cases}
\end{equation*}
As we will see, the decay estimates for Timoshenko-Cattaneo and Timoshenko-Fourier, depend on the properties of the function $\rho(\xi)$. In fact, this function $\rho(\xi)$ behaves like $\xi^4$ in the low frequency region $(|\xi| \leq 1)$ and like $\xi^{-2}$ near infinity, whenever $\chi_{0,\tau} = 0$ (resp, $\chi_{0} = 0$). Otherwise, if $\chi_{0,\tau} \neq 0$ (resp, $\chi_{0} \neq 0$), the function $\rho(\xi)$ behaves also like $\xi^4$ in the low frequency region but like $\xi^{-4}$ near infinity, which means that the dissipation in the hight frequency region is very weak and produces the \textit{regularity loss phenomenon}. It is known that this regularity loss causes some difficulties in the nonlinear cases, see for example \cite{ide2008, kawashima2008} for more details.

This paper is organized as follows. Section \ref{section2} is dedicated to state the problems. In section \ref{section3}, we will present the energy method in the Fourier space and the construction of the Lyapunov functionals. The main results, Theorems \ref{teocattaneo} and \ref{teofourier}  are formulated in Section \ref{section4}.


\section{Setting of the Problem}\label{section2}

In order to establish the decay rates of the Timoshenko systems \eqref{system0} and \eqref{system0'}, we have to transform the original problems to a first-order (in variable $t$) systems, defining new variables. Then, we apply the energy method in the Fourier space to prove some point wise estimates which will help in the proof of the decay estimates.


\subsection{The Cattaneo Model}\label{subsecC}

We consider the Timoshenko system with history and Cattaneo law. Using the change of variable, introduced in \cite{dafermos1970},
\begin{equation}\label{eta}
\eta(t,s,x):=\psi(t,x)-\psi(t-s,x),\qquad (t,x) \in  (0,\infty) \times \R, \quad s\geq 0,
\end{equation}
the system \eqref{system0}, can be rewritten as
\begin{equation}\label{system1}
\left\lbrace\begin{tabular}{l l}
$\rho_1\varphi _{tt}-k \left ( \varphi_x - \psi \right )_x = 0$ & in $ (0,\infty) \times \R$, \\
$\rho_2\psi_{tt}-a\psi_{xx} - m \displaystyle\int_0^{\infty}g(s)\eta_{xx}(s)ds - k\left ( \varphi_x - \psi\right )+\delta \theta_{x} =0$ & in $ (0,\infty) \times \R,$ \\
$\rho_3\theta_{t} +q_x +\delta \psi_{xt}=0$ & in $ (0,\infty) \times \R,$\\
$\tau q_t+\beta q +\theta_x=0$ & in $ (0,\infty) \times \R,$ \\
$\eta_t+\eta_s -\psi_t=0$ & in $ (0,\infty) \times \R,$ \\
$\eta(\cdot,0,\cdot)=0$ & in $ (0,\infty) \times \R,$
\end{tabular} \right.
\end{equation}
where $a = a(b,g)$ is a positive constant given by $(H_3)$ and operator $T\eta=-\eta_s$ is the usual operator defined in problems with history terms, see for instance \cite{sare2008, sare2009} and references therein. Here, the last two equations of system \eqref{system1} are obtained differentiating equation \eqref{eta}. We define also the initial data
\begin{gather*}
\begin{tabular}{l l l}
$\varphi(0,\cdot)=\varphi_0(\cdot)$,   & $\psi(0,\cdot)=\psi_0(\cdot)$,   & $\theta(0,\cdot)=\theta_0(\cdot)$, \\
$\varphi_t(0,\cdot)=\varphi_1(\cdot)$, & $\psi_t(0,\cdot)=\psi_1(\cdot)$, & $q(0,\cdot)=q_0(\cdot)$, \\
\end{tabular}\\
\eta (0,s,\cdot)= \psi_0(0,\cdot)-\psi_0(-s,\cdot).
\end{gather*}
Furthermore, we can rewrite the system \eqref{system1} by considering the following change of variables
\begin{equation*}
 u=\varphi_t,  \qquad  z=\psi_x,  \qquad y=\psi_t, \qquad v=\varphi_x-\psi.
\end{equation*}
Then, \eqref{system1}  takes the form
\begin{equation}\label{system2}
\left\lbrace \begin{tabular}{l}
$v_t-u_x+y=0$, \\
$\rho_1u_t-kv_x=0$, \\
$z_t-y_x =0$, \\
$\rho_2y_t -az_x-m\displaystyle\int_0^{\infty}g(s)\eta_{xx}(s)ds -kv+\delta \theta_x=0$,\\
$\rho_3\theta_t+q_x+\delta y_x=0$,\\
$\tau q_t +\beta q+\theta_x=0$, \\
$\eta_t+\eta_s-y=0$,\\
$\eta(\cdot,0,\cdot)=0$
\end{tabular}\right.
\end{equation}
Now, we define the solution of \eqref{system2} by the vector $U$, which is given by
\begin{equation*}
U(t,x)=(v,u,z,y,\theta,q,\eta)^{T}.
\end{equation*}
The initial condition can be written as
\begin{equation}\label{system20}
U_0(x)=U(0,x)=(v_0,u_0,z_0,y_0,\theta_0,q_0,\eta_0)^{T},
\end{equation}
where $u_0=\varphi_1,  z_0=\psi_{0,x}, y_0=\psi_1$, $v_0=\varphi_{0,x}-\psi_0$ and $\eta_0=\eta (0,s,\cdot)$ which is defined, as usual, in the history space $L^2_g(\R^+,H^1(\R))$, endowed with the norm
$$
||\eta||^2:=\int_{\R}\int_0^{\infty}g(s)|\eta_x(s)|^2dsdx.
$$


\subsection{The Fourier Model}\label{subsecF}

Similarly to Section \ref{subsecC}, we consider the Timoshenko system \eqref{system0} with history and the Fourier law, i.e, when $\tau = 0$.  Indeed, we can eliminate $q$ easily and obtain the following differential equation for $\theta$:
\begin{equation*}
\rho_3 \theta_t - \tilde{\beta} \theta_{xx} +\delta \psi _{xt} = 0,
\end{equation*}
where $\tilde{\beta}=\beta^{-1}>0$. Then, introducing $\eta$ as in the previous subsection, we have the differential equations
\begin{equation}\label{system3}
\left\lbrace\begin{tabular}{l l}
$\rho_1\varphi _{tt}-k \left ( \varphi_x - \psi \right )_x = 0$ & in $ (0,\infty) \times \R$, \\
$\rho_2\psi_{tt}-a\psi_{xx} - m \displaystyle\int_0^{\infty}g(s)\eta_{xx}(s)ds - k\left ( \varphi_x - \psi\right )+\delta \theta_{x} =0$ & in $ (0,\infty) \times \R,$ \\
$\rho_3\theta_{t} -\tilde{\beta}\theta_{xx} +\delta \psi_{xt}=0$ & in $ (0,\infty) \times \R,$\\
$\eta_t+\eta_s -\psi_t=0$ & in $ (0,\infty) \times \R.$ \\
$\eta(\cdot,0,\cdot)=0$ & in $ (0,\infty) \times \R.$
\end{tabular} \right.
\end{equation}
with initial data
\begin{equation*}
\begin{tabular}{l l l }
$\varphi(0,\cdot)=\varphi_0(\cdot)$,   & $\psi(0,\cdot)=\psi_0(\cdot)$,   & $\theta(0,\cdot)=\theta_0(\cdot)$, \\
$\varphi_t(0,\cdot)=\varphi_1(\cdot)$, & $\psi_t(0,\cdot)=\psi_1(\cdot)$, & $\eta (0,s,\cdot)= \psi(0,\cdot)-\psi(-s,\cdot)$.
\end{tabular}
\end{equation*}
As in the previous section, we can rewrite the system as a fist-order system, by defining the following variables
\begin{equation*}
 u=\varphi_t,  \qquad  z=\psi_x,  \qquad y=\psi_t, \qquad v=\varphi_x-\psi.
\end{equation*}
Then, \eqref{system3}  takes the form,
\begin{equation}\label{system4}
\left\lbrace \begin{tabular}{l}
$v_t-u_x+y=0$, \\
$\rho_1u_t-kv_x=0$, \\
$z_t-y_x =0$, \\
$\rho_2y_t -az_x-m\displaystyle\int_0^{\infty}g(s)\eta_{xx}(s)ds -kv+\delta \theta_x=0$,\\
$\rho_3\theta_t-\tilde{\beta}\theta_{xx}+\delta y_x=0$,\\
$\eta_t+\eta_s-y=0$.
\end{tabular}\right.
\end{equation}
We define the vector solution  $V$ of the system \eqref{system4}, as
\begin{equation*}
V(t,x)=(v,u,z,y,\theta,\eta)^{T}.
\end{equation*}
Thus, the initial condition can be written
\begin{equation}\label{system40}
V_0(x)=V(x,0)=(v_0,u_0,z_0,y_0,\theta_0,\eta_0)^{T},
\end{equation}
where $u_0=\varphi_1,  z_0=\psi_{0,x}, y_0=\psi_1$, $v_0=\varphi_{0,x}-\psi_0$ and $\eta_0=\eta (0,s,\cdot)$ defined in the history space given in the Cattaneo's version.

\section{The energy method in the frequency space}\label{section3}

This section is devoted to showing the relationship between the rate of decay of solutions and the new condition (see \cite{fatori2014})
\begin{equation*}
\chi_{0,\tau}=\left( \tau-\frac{\rho_1}{\rho_3 k}\right)\left( \rho_2 -\frac{b \rho_1}{k}\right) - \frac{\tau \rho_1 \delta^2}{\rho_3 k}.
\end{equation*}
For this reason, we will discuss two cases: the case where $\chi_{0,\tau}=0$ and the case where $\chi_{0,\tau}\not= 0$. Moreover, for the Timoshenko-Fourier model, i,e., when $\tau=0$, we consider the usual wave speeds propagation given by
\begin{equation*}
\chi_{0}= \rho_2 -\frac{b \rho_1}{k}.
\end{equation*}
In each case, we use a delicate energy method to build appropriate Lyapunov functionals in the Fourier space.


\subsection{The Timoshenko-Cattaneo Law}
We consider the Fourier image of the Timoshenko-Cattaneo model with history and we show that the heat damping induced by Cattaneo law and the past history are strong enough to stabilize the whole system. Thus, taking Fourier Transform in (\ref{system2}), we obtain the following integro-differential system:

\begin{align}
&\hat{v}_t-i\xi\hat{u}+\hat{y}=0, \label{e1}\\
&\rho_1\hat{u}_t-ik\xi\hat{v}=0, \label{e2}\\
&\hat{z}_t-i\xi\hat{y} =0, \label{e3}\\
&\rho_2\hat{y}_t -ia\xi\hat{z}+m\xi^2\displaystyle\int_0^{\infty}g(s)\hat{\eta}(s)ds -k\hat{v}+i\delta\xi\hat{\theta}=0,\label{e4}\\
&\rho_3\hat{\theta}_t+i\xi\hat{q}+i\delta\xi\hat{y}=0,\label{e5}\\
&\tau \hat{q}_t +\beta \hat{q}+i\xi\hat{\theta}=0, \label{e6}\\
&\hat{\eta}_t+\hat{\eta}_s-\hat{y}=0. \label{e7}
\end{align}
Here, the solution vector and initial data are given by $\hat{U}(\xi,t)=(\hat{v},\hat{u},\hat{z},\hat{y},\hat{\theta},\hat{q},\hat{\eta})^{T}$ and $\hat{U}(\xi,0)=\hat{U}_0(\xi)$, respectively. The energy functional associated to the above system is defined as:
\begin{equation}\label{energy1}
\hat{E}\left ( \xi,t\right) = \rho_1|\hat{u}|^{2}+\rho_2|\hat{y}|^{2}+\rho_3 |\hat{\theta}|^{2}+k|\hat{v} |^{2}+a|\hat{z} |^{2}+\tau|\hat{q}|^{2} +m\xi^2\int_{0}^{\infty}g(s)|\hat{\eta}(s)|^2ds.
\end{equation}
\begin{lem}\label{lem1}
The energy \eqref{energy1} satisfies the following estimate:
\begin{gather}\label{energydissip}
\frac{d}{dt}\hat{E}(\xi,t)\leq -2\beta |\hat{q}|^2 -k_1m\xi^2\int_0^{\infty}g(s)|\hat{\eta}(s)|^2ds,
\end{gather}
where the constant $k_1>0$ is given by $(H_2)$.
\end{lem}
\begin{proof}
Multiplying \eqref{e1} by $k\overline{\hat{v}}$,  \eqref{e2} by $\overline{\hat{u}}$, \eqref{e3}  by $a\overline{\hat{z}}$,  \eqref{e4}  by $\overline{\hat{y}}$, \eqref{e5} by $\overline{\hat{\theta}}$ and \eqref{e6} by $\overline{\hat{q}}$, adding and  taking real part, it follows that
\begin{equation}\label{e8}
\frac{1}{2}\frac{d}{dt}\left\lbrace \rho_1|\hat{u}|^{2}+\rho_2|\hat{y}|^{2}+\rho_3 |\hat{\theta}|^{2}+k|\hat{v} |^{2}+a|\hat{z} |^{2}+\tau|\hat{q}|^{2}\right\rbrace=-\beta |\hat{q}|^2 -Re\left(m \xi^2\int_0^{\infty} g(s)\eta(s)\overline{\hat{y}}ds\right).
\end{equation}
On the other hand, taking the conjugate of equation \eqref{e7}, multiplying the resulting equation by $g(s)\hat{\eta}(t,s,x)$ and integrating with respect to $s$, we obtain
\begin{equation}\label{e9}
Re\left(\int_0^{\infty}g(s)\hat{\eta}(s)\overline{\hat{y}}ds\right)=\frac{1}{2}\frac{d}{dt}\int_0^{\infty}g(s)|\hat{\eta}(s)|^2ds +\frac{1}{2}\int_0^{\infty}g(s)\frac{d}{ds}|\hat{\eta}(s)|^2ds.
\end{equation}
Integrating by parts the last term in the left hand side of \eqref{e9}, we have
\begin{equation*}
Re\left( \int_0^{\infty}g(s)\hat{\eta}(s)\overline{\hat{y}}ds\right)=\frac{1}{2}\frac{d}{dt}\int_0^{\infty}g(s)|\hat{\eta}(s)|^2ds -\frac{1}{2}\int_0^{\infty}g'(s)|\hat{\eta}(s)|^2ds.
\end{equation*}
Plugging the above equation in \eqref{e8}, it follows that
\begin{equation*}
\frac{d}{dt}\hat{E}(\xi,t)=-2\beta |\hat{q}|^2 +m\xi^2\int_0^{\infty}g'(s)|\eta(s)|^2ds.
\end{equation*}
Using $(H_2)$, we obtain \eqref{energydissip}.
\end{proof}
\vglue 0.3cm
\noindent With this energy dissipation in hands, the following questions arise:
\begin{center}
\textit{Does $\hat{E}(t) \rightarrow 0$ as $t \rightarrow \infty$? If it is the case, can we find the decay rate of $\hat{E}(t)$?}
\end{center}
The following Theorem provides a positive answer establishing the exponential decay of the integro-differential system \eqref{e1}-\eqref{e7}. This result  is a fundamental ingredient in the proof of our main results.
\begin{thm}\label{teo1}
Let
\begin{equation}\label{e21}
\chi_{0,\tau}=\left( \tau-\frac{\rho_1}{\rho_3 k}\right)\left( \rho_2 -\frac{b \rho_1}{k}\right) - \frac{\tau \rho_1 \delta^2}{\rho_3 k}.
\end{equation}
Then, for any $t \geq 0$ and $\xi \in \R$, the energy of system \eqref{e1}-\eqref{e7} satisfies
\begin{equation}\label{e22}
\hat{E}(\xi,t) \leq C e^{-\lambda \rho(\xi)}\hat{E}(0,\xi),
\end{equation}
where $C, \lambda$ are positive constants and the function $\rho(\cdot)$ is given by
\begin{equation}\label{e23}
\rho(\xi) = \begin{cases}
\dfrac{\xi^4}{(1+\xi^2)^{3}} & \text{if $\chi_{0,\tau}=0$}, \\
\\
\dfrac{\xi^4}{(1+\xi^2)^{4}} & \text{if $\chi_{0,\tau}\neq 0$}.
\end{cases}
\end{equation}
\end{thm}

Following the ideas contain in  \cite{said2013damping}, we construct some functionals to capture the dissipation of all the components of the vector solution. These functionals allow us to build an appropriate Lyapunov functional equivalent to the energy. The proof of Theorem \ref{teo1} is based on the following lemmas:

\begin{lem}\label{lem2}
Consider the functional
\begin{multline*}
J_1(\xi,t)=-\tau\rho_2Re\left(\hat{v}\overline{\hat{y}}\right)-\frac{a\tau\rho_1}{k}Re\left(\hat{z}\overline{\hat{u}}\right)+\frac{\delta\rho_1}{k}\left(\tau + \frac{1}{\delta^2}\left(\rho_2-\frac{b\rho_1}{k}+\tau b_0\rho_3\right)\right)Re(\hat{\theta}\overline{\hat{u}}) \\
 - \frac{\tau}{\delta}\left(\rho_2-\frac{b\rho_1}{k}+\tau b_0\rho_3\right) Re(\hat{v}\overline{\hat{q}})
\end{multline*}
Then, for any $\varepsilon>0$, $J_1$ satisfies
\begin{multline}\label{n5}
\frac{d}{dt}J_1(\xi,t)  +\tau k(1-\varepsilon)|\hat{v}|^2   \leq
\tau \rho_2|\hat{y}|^2 +C(\varepsilon)\xi^4\left|\int_0^{\infty}g(s)\hat{\eta}(s)ds\right|^2 \\
+\chi_{0,\tau} Re\left(i\xi\overline{\hat{u}}\hat{y}\right)
+\frac{1}{\delta}\left(  \chi_{0,\tau}+\tau b_0\rho_3\left(\tau -\frac{\rho_1}{\rho_3k}\right) \right) Re(i\xi\hat{q}\overline{\hat{u}}) \\
+\frac{\tau}{\delta}\left(\rho_2-\frac{b\rho_1}{k} + \tau b_0 \rho_3\right) Re(\hat{y}\overline{\hat{q}}) +C(\varepsilon)|\hat{q}|^2.
\end{multline}
where $C(\varepsilon)$ is a positive constant and $\chi_{0,\tau}$ is given by \eqref{propagation}.
\end{lem}
\begin{proof}
Multiplying \eqref{e1} by $-\rho_2 \overline{\hat{y}}$ and taking real part, we obtain
\begin{align*}
-\rho_2Re\left(\hat{v}_t\overline{\hat{y}}\right) + \rho_2Re\left(i\xi\hat{u}\overline{\hat{y}}\right)-\rho_2|\hat{y}|^2=0.
\end{align*}
Multiplying \eqref{e4} by $- \overline{\hat{v}}$ and taking real part, it follows that
\begin{align*}
-\rho_2Re\left(\hat{y}_t\overline{\hat{v}}\right) + aRe\left(i\xi\hat{z}\overline{\hat{v}}\right)+k|\hat{v}|^2-Re\left(m\xi^2\overline{\hat{v}}\int_0^{\infty}g(s)\hat{\eta}(s)ds\right)-\delta Re\left(i\xi\hat{\theta}\overline{\hat{v}}\right)=0.
\end{align*}
Adding the above identities,
\begin{multline}\label{e11}
-\rho_2\frac{d}{dt}Re\left(\hat{v}\overline{\hat{y}}\right)+k|\hat{v}|^2=\rho_2|\hat{y}|^2 - aRe\left(i\xi\hat{z}\overline{\hat{v}}\right)
+Re\left(m\xi^2\overline{\hat{v}}\int_0^{\infty}g(s)\hat{\eta}(s)ds\right) -\rho_2Re\left(i\xi\hat{u}\overline{\hat{y}}\right)+\delta Re\left(i\xi\hat{\theta}\overline{\hat{v}}\right).
\end{multline}
On the other hand, multiplying \eqref{e2} by $-\frac{a}{k}\overline{\hat{z}}$, \eqref{e3} by $-\frac{a\rho_1}{k}\overline{\hat{u}}$, adding the results and taking real part, it follows that
\begin{equation}\label{e12}
-\frac{a\rho_1}{k}\frac{d}{dt}Re\left(\hat{z}\overline{\hat{u}}\right) = -\frac{a\rho_1}{k}Re\left(i\xi\hat{y}\overline{\hat{u}}\right)-aRe\left(i\xi\hat{v}\overline{\hat{z}}\right).
\end{equation}
Moreover, multiplying \eqref{e2} by $\frac{\delta}{k}\overline{\hat{\theta}}$ and taking real part,
\begin{equation*}
\frac{\delta\rho_1}{k}Re(\hat{u}_t\overline{\hat{\theta}})-\delta Re(i\xi\hat{v}\overline{\hat{\theta}})=0.
\end{equation*}
Next, multiplying \eqref{e5} by $\frac{\delta\rho_1}{\rho_3 k}\overline{\hat{u}}$ and taking real part,
\begin{equation*}
\frac{\delta\rho_1}{k}Re(\hat{\theta}_t\overline{\hat{u}})+\frac{\delta\rho_1}{\rho_3 k}Re(i\xi\hat{q}\overline{\hat{u}})+\frac{\delta^2\rho_1}{\rho_3 k}Re(i\xi\hat{y}\overline{\hat{u}})=0.
\end{equation*}
Adding the above identities, we obtain
\begin{equation}\label{n1}
\frac{\delta\rho_1}{k}\frac{d}{dt}Re(\hat{\theta}\overline{\hat{u}})= -\frac{\delta\rho_1}{\rho_3 k}Re(i\xi\hat{q}\overline{\hat{u}})-\frac{\delta^2\rho_1}{\rho_3 k}Re(i\xi\hat{y}\overline{\hat{u}})+\delta Re(i\xi\hat{v}\overline{\hat{\theta}}).
\end{equation}
Furthermore, multiplying \eqref{e6} by $-\overline{\hat{v}}$ and taking real part,
\begin{equation*}
-\tau Re(\hat{q}_t\overline{\hat{v}})-\beta Re(\hat{q}\overline{\hat{v}})-Re(i\xi\hat{\theta}\overline{\hat{v}})=0.
\end{equation*}
Multiplying, \eqref{e1} by $-\tau \overline{\hat{q}}$ and taking real part,
\begin{equation*}
-\tau Re(\hat{v}_t\overline{\hat{q}})+\tau Re(i\xi\hat{u}\overline{\hat{q}})-\tau Re(\hat{y}\overline{\hat{q}})=0.
\end{equation*}
Adding the above identities, it follows that
\begin{equation}\label{n2}
- \tau\frac{d}{dt} Re(\hat{v}\overline{\hat{q}})=-\tau Re(i\xi\hat{u}\overline{\hat{q}})+\tau Re(\hat{y}\overline{\hat{q}})
+\beta Re(\hat{q}\overline{\hat{v}})+Re(i\xi\hat{\theta}\overline{\hat{v}}).
\end{equation}
Now, computing \eqref{e11}+\eqref{e12}+\eqref{n1}, we have
\begin{multline}\label{n3}
\frac{d}{dt}\left\lbrace -\rho_2Re\left(\hat{v}\overline{\hat{y}}\right)-\frac{a\rho_1}{k}Re\left(\hat{z}\overline{\hat{u}}\right) +\frac{\delta\rho_1}{k}Re(\hat{\theta}\overline{\hat{u}})\right\rbrace  +k|\hat{v}|^2
=\rho_2|\hat{y}|^2 +Re\left(m\xi^2\overline{\hat{v}}\int_0^{\infty}g(s)\hat{\eta}(s)ds\right)  \\
+\left(\rho_2 - \frac{a\rho_1}{k}-\frac{\delta^2\rho_1}{\rho_3 k}\right) Re\left(i\xi\overline{\hat{u}}\hat{y}\right) - \frac{\delta\rho_1}{\rho_3 k}Re(i\xi\hat{q}\overline{\hat{u}}).
\end{multline}
Computing $\tau \eqref{e11} + \tau\eqref{e12} + \underbrace{\left(\tau + \frac{1}{\delta^2} \left( \rho_2-\frac{b\rho_1}{k} +\tau b_0 \rho_3\right)\right)}_{\Gamma}\eqref{n1}$, we find that
\begin{multline*}
\frac{d}{dt}\left\lbrace -\tau\rho_2Re\left(\hat{v}\overline{\hat{y}}\right)-\frac{a\tau\rho_1}{k}Re\left(\hat{z}\overline{\hat{u}}\right)+ \frac{\Gamma\delta\rho_1}{k}Re(\hat{\theta}\overline{\hat{u}})\right\rbrace  +\tau k|\hat{v}|^2
= \tau \rho_2|\hat{y}|^2 + \tau m Re\left(\xi^2\overline{\hat{v}}\int_0^{\infty}g(s)\hat{\eta}(s)ds\right)  \\
+ \left[\tau\left(\rho_2 - \frac{a\rho_1}{k}\right)-\frac{\delta^2\rho_1}{\rho_3 k} \Gamma \right] Re\left(i\xi\overline{\hat{u}}\hat{y}\right)
+\tau \delta Re(i\xi\hat{\theta}\overline{\hat{v}})
- \frac{\delta\rho_1}{\rho_3 k} \Gamma Re(i\xi\hat{q}\overline{\hat{u}}) + \Gamma \delta Re(i\xi\overline{\hat{\theta}}\hat{v}).
\end{multline*}
Hence,
\begin{multline}\label{n4}
\frac{d}{dt}\left\lbrace -\tau\rho_2Re\left(\hat{v}\overline{\hat{y}}\right)-\frac{a\tau\rho_1}{k}Re\left(\hat{z}\overline{\hat{u}}\right)+ \frac{\Gamma\delta\rho_1}{k}Re(\hat{\theta}\overline{\hat{u}})\right\rbrace  +\tau k|\hat{v}|^2
= \tau \rho_2|\hat{y}|^2 + \tau m Re\left(\xi^2\overline{\hat{v}}\int_0^{\infty}g(s)\hat{\eta}(s)ds\right)  \\
+\left[\left(\tau-\frac{\rho_1}{\rho_3 k}\right)\left(\rho_2 - \frac{b\rho_1}{k}\right)-\frac{\delta^2\rho_1}{\rho_3 k} \tau \right] Re\left(i\xi\overline{\hat{u}}\hat{y}\right)
+\delta\left(\tau -\Gamma \right) Re(i\xi\hat{\theta}\overline{\hat{v}})
- \frac{\delta\rho_1}{\rho_3 k}\Gamma Re(i\xi\hat{q}\overline{\hat{u}}).
\end{multline}
Multiplying \eqref{n2} by $\Gamma_1=- \delta(\tau -\Gamma)$ and adding the result to \eqref{n4}, it follows that
\begin{multline*}
\frac{d}{dt}\left\lbrace -\tau\rho_2Re\left(\hat{v}\overline{\hat{y}}\right)-\frac{a\tau\rho_1}{k}Re\left(\hat{z}\overline{\hat{u}}\right) +\frac{\Gamma\delta\rho_1}{k}Re(\hat{\theta}\overline{\hat{u}}) - \tau \Gamma_1 Re(\hat{v}\overline{\hat{q}})\right\rbrace  +\tau k|\hat{v}|^2 \\
= \tau \rho_2|\hat{y}|^2 + \tau m Re\left(\xi^2\overline{\hat{v}}\int_0^{\infty}g(s)\hat{\eta}(s)ds\right)
+\left[\left(\tau-\frac{\rho_1}{\rho_3 k}\right)\left(\rho_2 - \frac{a\rho_1}{k}\right)-\frac{\delta^2\rho_1}{\rho_3 k} \tau \right] Re\left(i\xi\overline{\hat{u}}\hat{y}\right)
\\
+\delta\left(\tau -\Gamma \right) Re(i\xi\hat{\theta}\overline{\hat{v}})
- \frac{\delta\rho_1}{\rho_3 k} \Gamma Re(i\xi\hat{q}\overline{\hat{u}}) -\Gamma_1\tau Re(i\xi\hat{u}\overline{\hat{q}})+\Gamma_1\tau Re(\hat{y}\overline{\hat{q}})
\\
+\Gamma_1\beta Re(\hat{q}\overline{\hat{v}})+\Gamma_1 Re(i\xi\hat{\theta}\overline{\hat{v}}).
\end{multline*}
Hence,
\begin{multline*}
\frac{d}{dt}J_1(\xi,t)  +\tau k|\hat{v}|^2
= \tau \rho_2|\hat{y}|^2 + \tau m Re\left(\xi^2\overline{\hat{v}}\int_0^{\infty}g(s)\hat{\eta}(s)ds\right)  \\
+\chi_{0,\tau} Re\left(i\xi\overline{\hat{u}}\hat{y}\right)
+\left( -\delta\tau (\tau-\Gamma)  - \frac{\delta\rho_1\Gamma}{\rho_3 k}\right) Re(i\xi\hat{q}\overline{\hat{u}}) \\
-\delta \tau(\tau-\Gamma) Re(\hat{y}\overline{\hat{q}}) -\delta \beta (\tau-\Gamma) Re(\hat{q}\overline{\hat{v}}).
\end{multline*}
Applying Young inequality, \eqref{n5} follows.

\end{proof}


\begin{lem}\label{lem3}
Consider the functional
\begin{equation}\label{e13}
J_2(\xi,t)=\rho_1Re\left(i\xi\hat{v}\overline{\hat{u}}\right)+ \rho_2Re\left(i\xi\hat{y}\overline{\hat{z}}\right)+\delta\tau Re(i\xi\hat{z}\overline{\hat{q}})
\end{equation}
For any $\varepsilon >0$, the estimate
\begin{multline}\label{e14}
\frac{d}{dt}J_2(\xi,t) +\rho_1(1-\varepsilon)\xi^2|\hat{u}|^2+a(1-\varepsilon)\xi^2|\hat{z}|^2\leq C(\varepsilon)(1+\xi^2)|\hat{v}|^2 +C(\varepsilon)(1+\xi^2)|\hat{y}|^2 \\
+C(\varepsilon)(1+\xi^2)|\hat{q}|^2 +C(\varepsilon)\xi^4\left|\int_0^{\infty}g(s)\hat{\eta}(s)ds\right|^2
\end{multline}
is satisfied.
\end{lem}

\begin{proof}
Multiplying \eqref{e1} by $i\rho_1\xi \overline{\hat{u}}$ and taking real part,
\begin{align*}
\rho_1Re\left(i\xi\hat{v}_t\overline{\hat{u}}\right) +\rho_1\xi^2|\hat{u}|^2+\rho_1Re\left(i\xi\hat{y}\overline{\hat{u}}\right)=0.
\end{align*}
Multiplying \eqref{e2} by $-i\xi \overline{\hat{v}}$ and taking real part,
\begin{align*}
-\rho_1Re\left(i\xi\hat{u}_t\overline{\hat{v}}\right) -k\xi^2|\hat{v}|^2=0.
\end{align*}
Adding the above identities, we obtain
\begin{align}\label{e15}
\rho_1\frac{d}{dt}Re\left(i\xi\hat{v}\overline{\hat{u}}\right)+\rho_1\xi^2|\hat{u}|^2=k\xi^2|\hat{v}|^2 -\rho_1Re\left(i\xi\hat{y}\overline{\hat{u}}\right).
\end{align}
Moreover, multiplying \eqref{e3} by $-i\rho_2\xi \overline{\hat{y}}$ and taking real part,
\begin{align*}
-\rho_2Re\left(i\xi\hat{z}_t\overline{\hat{y}}\right) -\rho_2\xi^2|\hat{y}|^2=0.
\end{align*}
Multiplying \eqref{e4} by $i\xi \overline{\hat{z}}$ and taking real part,
\begin{align*}
\rho_2Re\left(i\xi\hat{y}_t\overline{\hat{z}}\right) +a\xi^2|\hat{z}|^2 -kRe\left(i\xi\hat{v}\overline{\hat{z}}\right)+m Re\left(i\xi^3\overline{\hat{z}}\int_0^{\infty}g(s)\hat{\eta}(s)ds\right) -\delta Re\left(\xi^2\hat{\theta}\overline{\hat{z}}\right)=0.
\end{align*}
Adding the above identities,
\begin{equation}\label{e16}
\rho_2\frac{d}{dt}Re\left(i\xi\hat{y}\overline{\hat{z}}\right) +a\xi^2|\hat{z}|^2= \rho_2\xi^2|\hat{y}|^2+kRe\left(i\xi\hat{v}\overline{\hat{z}}\right)-m Re\left(i\xi^3\overline{\hat{z}}\int_0^{\infty}g(s)\hat{\eta}(s)ds\right) +\delta Re\left(\xi^2\hat{\theta}\overline{\hat{z}}\right).
\end{equation}
Multiplying \eqref{e6} by $-i\delta\xi \overline{\hat{z}}$ and taking real part,
\begin{equation*}
-\delta \tau Re(i\xi \hat{q}_t\overline{\hat{z}})-\beta \delta Re (i\xi\hat{q}\overline{\hat{z}})+\delta Re(\xi^2\hat{\theta}\overline{\hat{z}})=0.
\end{equation*}
Multiplying  \eqref{e3} by $i\delta\tau\xi \overline{\hat{q}}$ and taking real part,
\begin{equation*}
\delta \tau Re(i\xi\hat{z}_t\overline{\hat{q}})+\delta \tau Re(\xi^2\hat{y}\overline{\hat{q}})=0.
\end{equation*}
Adding the above identities, we obtain
\begin{equation}\label{n6}
\delta\tau\frac{d}{dt}  Re(i\xi\hat{z}\overline{\hat{q}}) = -\delta \tau Re(\xi^2\hat{y}\overline{\hat{q}}) + \beta \delta Re (i\xi\hat{q}\overline{\hat{z}})-\delta Re(\xi^2\hat{\theta}\overline{\hat{z}}).
\end{equation}
Computing \eqref{e15} $+$ \eqref{e16} $+$ \eqref{n6}, we find that
\begin{multline*}
\frac{d}{dt}\left\lbrace \rho_1Re\left(i\xi\hat{v}\overline{\hat{u}}\right)+ \rho_2Re\left(i\xi\hat{y}\overline{\hat{z}}\right)+\delta\tau Re(i\xi\hat{z}\overline{\hat{q}}) \right\rbrace +a\xi^2|\hat{z}|^2+\rho_1\xi^2|\hat{u}|^2= \rho_2\xi^2|\hat{y}|^2+kRe\left(i\xi\hat{v}\overline{\hat{z}}\right) \\
-m Re\left(i\xi^3\overline{\hat{z}}\int_0^{\infty}g(s)\hat{\eta}(s)ds\right)
-\delta \tau Re(\xi^2\hat{y}\overline{\hat{q}}) + \beta \delta Re (i\xi\hat{q}\overline{\hat{z}})+ k\xi^2|\hat{v}|^2 -\rho_1Re\left(i\xi\hat{y}\overline{\hat{u}}\right).
\end{multline*}
Hence,
\begin{multline*}
\frac{d}{dt}J_2(\xi,t)+a\xi^2|\hat{z}|^2+\rho_1\xi^2|\hat{u}|^2 \leq  \rho_2\xi^2|\hat{y}|^2+k|\xi||\hat{v}||\hat{z}| \\
+m|\xi|^3|\hat{z}|\left|\int_0^{\infty}g(s)\hat{\eta}(s)ds\right|
+\delta \tau \xi^2|\hat{y}||\hat{q}| + \beta \delta |\xi||\hat{q}||\hat{z}|+ k\xi^2|\hat{v}|^2 +\rho_1|\xi||\hat{y}||\hat{u}|
\end{multline*}
applying Young's inequality, we obtain \eqref{e14}.
\end{proof}

\begin{lem}\label{lem4}
Consider the functional
\begin{equation*}
J_3(\xi,t)=-\rho_2Re\left( \xi^2 \overline{\hat{y}}\int_0^{\infty}g(s)\hat{\eta}(s)ds\right).
\end{equation*}
Then, for any $\varepsilon>0$,  the following estimate
\begin{multline}\label{newe20}
\frac{d}{dt}J_3(\xi,t) +\rho_2b_0(1-\varepsilon)\xi^2|\hat{y}|^2 \leq C(\varepsilon) \xi^2  \int_0^{\infty}g(s)|\hat{\eta}(s)|^2ds +a |\xi|^3 |\hat{z}|\left| \int_0^{\infty}g(s)\hat{\eta}(s)ds\right| \\
+m\xi^4\left| \int_0^{\infty}g(s)\hat{\eta}(s)ds\right|^2 +k\xi^2 |\hat{v}| \left| \int_0^{\infty}g(s)\hat{\eta}(s)ds\right| +\delta |\xi|^3 |\hat{\theta}| \left| \int_0^{\infty}g(s)\hat{\eta}(s)ds\right|
\end{multline}
holds.
\end{lem}
\begin{proof}
Multiplying \eqref{e4} by $-\xi^2g(s) \overline{\hat{\eta}}$ and  taking the integration with respect to $s$ for the real parts, it follows that
\begin{multline*}
-\rho_2  Re\left(\xi^2 \hat{y}_t\int_0^{\infty}g(s)\overline{\hat{\eta}}(s)ds\right) + a Re\left(i \xi^3 \hat{z}\int_0^{\infty}g(s)\overline{\hat{\eta}}(s)ds\right) \\
-m Re\left( \xi^4 \int_0^{\infty}g(s)\hat{\eta}(s)ds\int_0^{\infty}g(s)\overline{\hat{\eta}}(s)ds\right)  + k Re\left( \xi^2 \hat{v}\int_0^{\infty}g(s)\overline{\hat{\eta}}(s)ds\right)  \\
- \delta Re\left(i \xi^3 \hat{\theta}\int_0^{\infty}g(s)\overline{\hat{\eta}}(s)ds\right)=0.
\end{multline*}
Multiplying \eqref{e7} by $-\rho_2\xi^2g(s) \overline{\hat{y}}$ and  taking the integration with respect to $s$ for the real parts, it follows that
\begin{align*}
-\rho_2  Re\left(\xi^2 \overline{\hat{y}}\int_0^{\infty}g(s)\hat{\eta}_t(s)ds\right)   -\rho_2  Re\left(\xi^2 \overline{\hat{y}}\int_0^{\infty}g(s)\hat{\eta}_s(s)ds\right)+\rho_2\xi^2\int_0^{\infty}g(s)ds |\hat{y}|^2=0.
\end{align*}
Adding the above identities, we obtain that
\begin{multline*}
\frac{d}{dt}K_3(\xi,t) +\rho_2\xi^2b_0|\hat{y}|^2 = \rho_2Re\left( \xi^2 \overline{\hat{y}}\int_0^{\infty}g(s)\hat{\eta}_s(s)ds\right) -aRe\left(i \xi^3 \hat{z}\int_0^{\infty}g(s)\overline{\hat{\eta}}(s)ds\right) \\
+m\xi^4\left| \int_0^{\infty}g(s)\hat{\eta}(s)ds\right|^2 - k Re\left( \xi^2 \hat{v}\int_0^{\infty}g(s)\overline{\hat{\eta}}(s)ds\right) +\delta Re\left(i \xi^3 \hat{\theta}\int_0^{\infty}g(s)\overline{\hat{\eta}}(s)ds\right).
\end{multline*}
Now, integrating by parts the first term on the right-hand side of the above equality, we get
\begin{multline*}
\frac{d}{dt}K_3(\xi,t) +\rho_2\xi^2b_0|\hat{y}|^2 = -\rho_2Re\left( \xi^2 \overline{\hat{y}}\int_0^{\infty}g'(s)\hat{\eta}(s)ds\right) -aRe\left(i \xi^3 \hat{z}\int_0^{\infty}g(s)\overline{\hat{\eta}}(s)ds\right) \\
+m\xi^4\left| \int_0^{\infty}g(s)\hat{\eta}(s)ds\right|^2 - k Re\left( \xi^2 \hat{v}\int_0^{\infty}g(s)\overline{\hat{\eta}}(s)ds\right) +\delta Re\left(i \xi^3 \hat{\theta}\int_0^{\infty}g(s)\overline{\hat{\eta}}(s)ds\right).
\end{multline*}
Hence,
\begin{multline*}
\frac{d}{dt}K_3(\xi,t) +\rho_2\xi^2b_0|\hat{y}|^2 \leq \rho_2 \xi^2|\hat{y}|  \left| \int_0^{\infty}g'(s)\hat{\eta}(s)ds\right| +a |\xi^3| |\hat{z}|\left| \int_0^{\infty}g(s)\overline{\hat{\eta}}(s)ds\right| \\
+m\xi^4\left| \int_0^{\infty}g(s)\hat{\eta}(s)ds\right|^2 +k\xi^2 |\hat{v}| \left| \int_0^{\infty}g(s)\overline{\hat{\eta}}(s)ds\right| +\delta |\xi|^3 |\hat{\theta}| \left| \int_0^{\infty}g(s)\overline{\hat{\eta}}(s)ds\right|.
\end{multline*}
Young inequality yields
\begin{multline*}
\frac{d}{dt}K_3(\xi,t) +\rho_2\xi^2b_0(1-\varepsilon)|\hat{y}|^2 \leq C(\varepsilon) \xi^2  \int_0^{\infty}g'(s)|\hat{\eta}(s)|^2ds  +a |\xi^3| |\hat{z}|\left| \int_0^{\infty}g(s)\overline{\hat{\eta}}(s)ds\right| \\
+m\xi^4\left| \int_0^{\infty}g(s)\hat{\eta}(s)ds\right|^2 +k\xi^2 |\hat{v}| \left| \int_0^{\infty}g(s)\overline{\hat{\eta}}(s)ds\right| +\delta |\xi|^3 |\hat{\theta}| \left| \int_0^{\infty}g(s)\overline{\hat{\eta}}(s)ds\right|.
\end{multline*}
Using the hypothesis on $g'$, the result follows.
\end{proof}
\begin{lem}\label{lem5}
Consider the functional
\begin{equation}\label{e19}
J_4(\xi,t)=\tau\rho_3Re\left(i\xi\hat{\theta}\overline{\hat{q}}\right).
\end{equation}
For any $\varepsilon>0$, the following estimate
\begin{align}\label{e20}
\frac{d}{dt}J_4(\xi,t) +\rho_3(1 -\varepsilon)\xi^2|\hat{\theta}|^2 \leq \tau\delta \xi^2|\hat{y}||\hat{q}|+C(\varepsilon)(1+\xi^2)|\hat{q}|^2
\end{align}
holds.
\end{lem}
\begin{proof}
Multiplying \eqref{e5} by $i\tau\xi \overline{\hat{q}}$ and taking real part,
\begin{align*}
\tau\rho_3Re\left(i\xi \hat{\theta}_t\overline{\hat{q}}\right)- \tau\xi^2|\hat{q}| ^2-\tau \delta Re\left(\xi^2\hat{y}\overline{\hat{q}}\right)=0.
\end{align*}
Multiplying \eqref{e6} by $-i\rho_3\xi \overline{\hat{\theta}}$ and taking real part,
\begin{align*}
-\tau\rho_3Re\left(i\xi \hat{q}_t\overline{\hat{\theta}}\right)-\beta\rho_3Re\left(i\xi\hat{q}\overline{\hat{\theta}}\right) +\rho_3\xi^2|\hat{\theta}| ^2=0.
\end{align*}
Adding up the above identities, it follows that
\begin{align*}
\frac{d}{dt}J_4(\xi,t) +\rho_3\xi^2|\hat{\theta}| ^2\leq  \tau\xi^2|\hat{q}| ^2+\tau \delta\xi^2|\hat{y}||\hat{q}| +\beta\rho_3|\xi||\hat{q}||\hat{\theta}|.
\end{align*}
Applying Young's inequality, the result follows.
\end{proof}
\vglue 0.3cm


\begin{proof}[ \textbf{\large Proof of Theorem \ref{teo1}}]
In order to make the proof  clear, we will consider several cases:
\vglue 0.2 cm
\noindent \textbf{I. Case $\boldsymbol{\chi_{0,\tau}=0}$:}. Consider the expressions
\[
\lambda_1\xi^2J_1(\xi,t), \quad \lambda_2\dfrac{\xi^2}{1+\xi^2}J_2(\xi,t), \quad \lambda_3J_3(\xi,t)
\]
where $\lambda_1$, $\lambda_2$ and  $\lambda_3$ are positive constants to be fixed later. Thus, Lemmas \ref{lem2} and \ref{lem3} imply that
\begin{multline}\label{e24}
\frac{d}{dt}\left\lbrace \lambda_1\xi^2 J_1(\xi,t)+\lambda_2\dfrac{\xi^2}{1+\xi^2}J_2(\xi,t) \right\rbrace + k\left[\lambda_1\tau(1-\varepsilon)-C(\varepsilon)\lambda_2\right]\xi^2|\hat{v}|^2 +\rho_1(1-2\varepsilon)\lambda_2\dfrac{\xi^4}{1+\xi^2}|\hat{u}|^2 \\
 +a(1-\varepsilon)\lambda_2\dfrac{\xi^4}{1+\xi^2}|\hat{z}|^2 \\
\leq C(\varepsilon,\lambda_1,\lambda_2)\xi^2|\hat{y}|^2 +C(\varepsilon,\lambda_1,\lambda_2)\xi^6\left|\int_0^{\infty}g(s)\hat{\eta}(s)ds\right|^2
 +C(\varepsilon,\lambda_1,\lambda_2)\xi^2(1+\xi^2)|\hat{q}|^2.
\end{multline}
Furthermore, Applying Young's inequality in equation \eqref{newe20} of Lemma \ref{lem4}, it follows that
\begin{multline*}
\frac{d}{dt}\lambda_3J_3(\xi,t) +\rho_2\lambda_3b_0 (1 -\varepsilon)\xi^2|\hat{y}|^2\leq C(\varepsilon,\lambda_3)\xi^2\int_0^{\infty}g(s)|\hat{\eta}(s)|^2ds+ a\lambda_2\varepsilon\frac{\xi^4}{1+\xi^2}|\hat{z}|^2 \\
+ C(\varepsilon,\lambda_2,\lambda_3)\xi^2(1+\xi^2)\left|\int_0^{\infty}g(s)\hat{\eta}(s)ds\right|^2 +\lambda_3m\xi^4\left|\int_0^{\infty}g(s)\hat{\eta}(s)ds\right|^2+k\lambda_1\tau\varepsilon\xi^2|\hat{v}|^2 \\
+C(\varepsilon,\lambda_1,\lambda_3)\xi^2\left|\int_0^{\infty}g(s)\hat{\eta}(s)ds\right|^2+ \lambda_3\delta |\xi|^3|\hat{\theta}|\left|\int_0^{\infty}g(s)\hat{\eta}(s)ds\right|.
\end{multline*}
Thus,
\begin{multline}\label{e25}
\frac{d}{dt}\lambda_3J_3(\xi,t) +\rho_2\lambda_3b_0 (1 -\varepsilon)\xi^2|\hat{y}|^2\leq C(\varepsilon,\lambda_3)\xi^2\int_0^{\infty}g(s)|\hat{\eta}(s)|^2ds+ a\lambda_2\varepsilon\frac{\xi^4}{1+\xi^2}|\hat{z}|^2 +k\lambda_1\tau\varepsilon\xi^2|\hat{v}|^2\\
+ C(\varepsilon,\lambda_1,\lambda_2,\lambda_3)\xi^2(1+\xi^2)\left|\int_0^{\infty}g(s)\hat{\eta}(s)ds\right|^2 + \lambda_3\delta |\xi|^3|\hat{\theta}|\left|\int_0^{\infty}g(s)\hat{\eta}(s)ds\right|.
\end{multline}
Computing \eqref{e24} $+$ \eqref{e25}, we obtain
\begin{multline*}
\frac{d}{dt}\left\lbrace \lambda_1\xi^2 J_1(\xi,t)+\lambda_2\dfrac{\xi^2}{1+\xi^2}J_2(\xi,t)+\lambda_3J_3(\xi,t) \right\rbrace + k\left[\lambda_1\tau(1-2\varepsilon)-C(\varepsilon)\lambda_2\right]\xi^2|\hat{v}|^2 +\rho_1(1-2\varepsilon)\lambda_2\dfrac{\xi^4}{1+\xi^2}|\hat{u}|^2  \\+a(1-2\varepsilon)\lambda_2\dfrac{\xi^4}{1+\xi^2}|\hat{z}|^2 +\rho_2\left[\lambda_3b_0 (1 -\varepsilon)-C(\varepsilon,\lambda_1,\lambda_2)\right]\xi^2|\hat{y}|^2 \\
\leq  C(\varepsilon,\lambda_3)\xi^2\int_0^{\infty}g(s)|\hat{\eta}(s)|^2ds +C(\varepsilon,\lambda_1,\lambda_2,\lambda_3)\xi^2(1+\xi^2)^2\left|\int_0^{\infty}g(s)\hat{\eta}(s)ds\right|^2  \\
+C(\varepsilon,\lambda_1,\lambda_2,\lambda_3)\xi^2(1+\xi^2)|\hat{q}|^2 + \lambda_3\delta |\xi|^3|\hat{\theta}|\left|\int_0^{\infty}g(s)\hat{\eta}(s)ds\right|.
\end{multline*}
Now, consider the following functional
\begin{equation*}
\LL_1(\xi,t)= \lambda_1\xi^2 J_1(\xi,t)+\lambda_2\dfrac{\xi^2}{1+\xi^2}J_2(\xi,t)+\lambda_3J_3(\xi,t)+ J_4(\xi,t).
\end{equation*}
From Lemma \ref{lem5} and using Young inequality, it yields that
\begin{multline*}
\frac{d}{dt} \LL_1(\xi,t)  + k\left[\lambda_1\tau(1-2\varepsilon)-C(\varepsilon)\lambda_2\right]\xi^2|\hat{v}|^2 +\rho_1(1-2\varepsilon)\lambda_2\dfrac{\xi^4}{1+\xi^2}|\hat{u}|^2  \\+a(1-2\varepsilon)\lambda_2\dfrac{\xi^4}{1+\xi^2}|\hat{z}|^2 +\rho_2\left[\lambda_3b_0 (1 -2\varepsilon)-C(\varepsilon,\lambda_1,\lambda_2)\right]\xi^2|\hat{y}|^2 +\rho_3(1 -2\varepsilon)\xi^2|\hat{\theta}|^2\\
\leq  C(\varepsilon,\lambda_3)\xi^2\int_0^{\infty}g(s)|\hat{\eta}(s)|^2ds +C(\varepsilon,\lambda_1,\lambda_2,\lambda_3)\xi^2(1+\xi^2)^2\left|\int_0^{\infty}g(s)\hat{\eta}(s)ds\right|^2  \\
+C(\varepsilon,\lambda_1,\lambda_2,\lambda_3)(1+\xi^2)^2|\hat{q}|^2.
\end{multline*}
Now, using the following inequality:
\begin{align}
\left|\int_0^{\infty}g(s)\hat{\eta}(s)ds\right|^2 &= \left|\int_0^{\infty}g^{\frac{1}{2}}(s)g^{\frac{1}{2}}(s)\hat{\eta}(s)ds\right|^2 \notag\\
&\leq \left|\left(\int_0^{\infty}g(s)ds\right)^{\frac{1}{2}}\left(\int_0^{\infty}g(s)|\hat{\eta}(s)|^2ds\right)^{\frac{1}{2}}\right|^2 \label{ineq}\\
&= b_0 \int_0^{\infty}g(s)|\hat{\eta}(s)|^2ds \notag,
\end{align}
we obtain that
\begin{multline}\label{e26}
\frac{d}{dt} \LL_1(\xi,t)  + k\left[\lambda_1\tau(1-2\varepsilon)-C(\varepsilon)\lambda_2\right]\xi^2|\hat{v}|^2 +\rho_1(1-2\varepsilon)\lambda_2\dfrac{\xi^4}{1+\xi^2}|\hat{u}|^2  \\+a(1-2\varepsilon)\lambda_2\dfrac{\xi^4}{1+\xi^2}|\hat{z}|^2 +\rho_2\left[\lambda_3b_0 (1 -2\varepsilon)-C(\varepsilon,\lambda_1,\lambda_2)\right]\xi^2|\hat{y}|^2 +\rho_3(1 -2\varepsilon)\xi^2|\hat{\theta}|^2\\
\leq  C_1(1+b_0)\xi^2(1+\xi^2)^2\int_0^{\infty}g(s)|\hat{\eta}(s)|^2ds   +C_1(1+\xi^2)^2|\hat{q}|^2,
\end{multline}
where $C_1$ is a  positive constant that depends on $\varepsilon$ and $\lambda_j$ for $j=1,2,3$. Then, we can choose the constants to make all the coefficients in the right side in \eqref{e26} positive. First, let us fix $\varepsilon$ such that $\varepsilon< \frac12.$ Thus, we can take first choose $\lambda_2>0$ and
\begin{equation}\label{constant1}
\lambda_1 >\frac{C(\varepsilon)\lambda_2}{\tau(1-2\varepsilon)}, \quad \lambda_3 >\frac{C(\varepsilon,\lambda_1,\lambda_2)}{b_0(1-2\varepsilon)}.
\end{equation}
Then, from \eqref{constant1} and some trivial inequalities such as
\begin{equation}\label{trivial}
\frac{\xi^2}{1+\xi^2}\leq 1 \quad \text{and} \quad \frac{1}{1+\xi^2}\leq 1,
\end{equation}
we can deduce the existence of a positive constant $M_1$ such that
\begin{multline}\label{e28}
\frac{d}{dt}\LL_1(\xi,t) \leq -M_1 \frac{\xi^4}{1+\xi^2}\left\lbrace k|\hat{v}|^2 +\rho_1|\hat{u}|^2
+a|\hat{z}|^2 + \rho_3|\hat{\theta}|^2 +\rho_2|\hat{y}|^2\right\rbrace \\
 + C_1(1+b_0)\xi^2(1+\xi^2)^2\int_0^{\infty}g(s)|\hat{\eta}(s)|^2ds   +C_1(1+\xi^2)^2|\hat{q}|^2.
\end{multline}
Finally, we define the following Lyapunov functional:
\begin{equation}\label{f1}
\LL(\xi,t) = \LL_1(\xi,t,t)+N(1+\xi^2)^{2}\hat{E}(\xi,t),
\end{equation}
where $N$ is a positive constant to be fixed later. Note that the definition of $\LL_1$ together with inequality \eqref{ineq}, imply that
\begin{align*}
\left|\LL_1(\xi,t)\right| &\leq M_2 \left\lbrace |J_1(\xi,t)|+|J_2(\xi,t)|+|J_3(\xi,t)|+ |J_4(\xi,t)| \right\rbrace \\
& \leq M_2 (1+\xi^2)^{ 2}\hat{E}(\xi,t).
\end{align*}
Hence, we obtain
\begin{equation}\label{e29}
(N-M_2)(1+\xi^2)^{2}\hat{E}(\xi,t) \leq \LL(\xi,t)\leq (N+M_2)(1+\xi^2)^{2} \hat{E}(\xi,t).
\end{equation}
On the other hand, taking the derivative of $\LL$ with respect to $t$ and using the estimates \eqref{e28} and Lemma \ref{lem1}, it follows that
\begin{multline*}
\frac{d}{dt}\LL(\xi,t) \leq -M_1 \frac{\xi^4}{1+\xi^2}\left\lbrace k|\hat{v}|^2 +\rho_1|\hat{u}|^2
+a|\hat{z}|^2 + \rho_3|\hat{\theta}|^2 +\rho_2|\hat{y}|^2\right\rbrace  \\
 -\left(2N\beta-C_1\right)(1+\xi^2)^{2}|\hat{q}|^2 -\left( k_1Nm - C_1 (1+b_0)\right)(1+\xi^2)^{2}\xi^2\int_0^{\infty}g(s)|\hat{\eta}(s)|^2ds.
\end{multline*}
Now, choosing $N$ such that $N\geq \max\left\lbrace M_2,\dfrac{C_1}{2\beta}, \dfrac{C_1(1+b_0)}{k_1m}\right\rbrace$ and using the inequality $(1+\xi^2)^{2}\geq \dfrac{\xi^4}{1+\xi^2}$, there exists  a positive constant $M_3$ such that
\begin{equation*}
\frac{d}{dt}\LL(\xi,t) \leq - M_3\frac{\xi^4}{1+\xi^2} \hat{E}(\xi,t).
\end{equation*}
Estimate $(\ref{e29})$ implies that
\begin{equation*}
\frac{d}{dt}\LL(\xi,t) \leq - \Gamma\frac{\xi^4}{(1+\xi^2)^{3}} \LL(\xi,t)
\end{equation*}
where $\Gamma=\dfrac{M_3}{N+M_2}$. By Gronwall's inequality, it  follows that
\begin{equation*}
\LL(\xi,t) \leq  e^{-\Gamma \rho(\xi)t}\LL(\xi,0), \qquad \rho(\xi)=\frac{\xi^4}{(1+\xi^2)^{3}}.
\end{equation*}
Again by using $(\ref{e29})$, we have that
\begin{align*}
\hat{E}(\xi,t) \leq C  e^{-\Gamma \rho(\xi)t}\hat{E}(\xi,0), \quad \text{where}\quad C= \frac{N+M_2}{N-M_2}> 0.
\end{align*}
\vglue 0.3 cm
\noindent\textbf{II. Case $\boldsymbol{\chi_{0,\tau}\neq 0}$:} Similar to previous case, we introduce positive constants $\gamma_1$, $\gamma_2$, $\gamma_3$ and $\gamma_4$ that will be fixed later. Next, we estimate the following terms by applying Young's
inequality,
\begin{align*}
\left| \chi_{0,\tau} Re(i\xi\hat{u}\overline{\hat{y}})\right| &\leq \frac{\rho_1\gamma_2\varepsilon}{2\gamma_1}\frac{\xi^2}{1+\xi^2}|\hat{u}|^2+C(\varepsilon,\gamma_1,\gamma_2)(1+\xi^2)|\hat{y}|^2, \\
\frac{1}{\delta}\left| \left(\chi_{0,\tau} +\tau b_0 \rho_3 \left( \tau - \frac{\rho_1}{\rho_3 k}\right)\right)Re(i\xi \hat{q}\overline{\hat{u}})\right| &\leq \frac{\rho_1\gamma_2\varepsilon}{2\gamma_1}\frac{\xi^2}{1+\xi^2}|\hat{u}|^2+C(\varepsilon,\gamma_1,\gamma_2)(1+\xi^2)|\hat{q}|^2.
\end{align*}
Hence, from Lemma \ref{lem2}, we can rewrite \eqref{n5} as
\begin{multline}\label{e31}
\frac{d}{dt}J_1(\xi,t) + k\tau(1-\varepsilon)|\hat{v}|^2 \leq  \frac{\rho_1\gamma_2\varepsilon}{\gamma_1}\frac{\xi^2}{1+\xi^2}|\hat{u}|^2+C(\varepsilon,\gamma_1,\gamma_2)(1+\xi^2)|\hat{y}|^2 \\
+C(\varepsilon,\gamma_1,\gamma_2)(1+\xi^2)|\hat{q}|^2+C(\varepsilon)\xi^4\left|\int_0^{\infty}g(s)\hat{\eta}(s)ds\right|^2.
\end{multline}
Thus, the Lemma \ref{lem3} and \eqref{e31} imply that
\begin{multline}\label{e32}
\frac{d}{dt}\left\lbrace\gamma_1\frac{\xi^2}{1+\xi^2} J_1(\xi,t)+\gamma_2\dfrac{\xi^2}{(1+\xi^2)^2}J_2(\xi,t) \right\rbrace + k\left[\gamma_1\tau(1-\varepsilon)-C(\varepsilon)\gamma_2\right]\frac{\xi^2}{1+\xi^2}|\hat{v}|^2 +\rho_1(1-2\varepsilon)\gamma_2\dfrac{\xi^4}{(1+\xi^2)^2}|\hat{u}|^2 \\
 +a(1-\varepsilon)\gamma_2\dfrac{\xi^4}{(1+\xi^2)^2}|\hat{z}|^2 \\
\leq C(\varepsilon,\gamma_1,\gamma_2)\xi^2|\hat{y}|^2
+C(\varepsilon,\gamma_1,\gamma_2)\xi^2|\hat{q}|^2+C(\varepsilon,\gamma_1)\frac{\xi^6}{1+\xi^2}\left|\int_0^{\infty}g(s)\hat{\eta}(s)ds\right|^2 \\
+C(\varepsilon,\gamma_2)\frac{\xi^2}{1+\xi^2}|\hat{y}|^2 +C(\varepsilon,\gamma_2)\frac{\xi^2}{1+\xi^2}|\hat{q}|^2+C(\varepsilon,\gamma_2)\frac{\xi^6}{(1+\xi^2)^2}\left|\int_0^{\infty}g(s)\hat{\eta}(s)ds\right|^2 \\
\leq C(\varepsilon,\gamma_1,\gamma_2)\xi^2|\hat{y}|^2
+C(\varepsilon,\gamma_1,\gamma_2)\xi^2|\hat{q}|^2+C(\varepsilon,\gamma_1,\gamma_2)(1+\xi^2)\xi^2\left|\int_0^{\infty}g(s)\hat{\eta}(s)ds\right|^2.
\end{multline}
Furthermore, applying Young's inequality in the equation \eqref{newe20} of the Lemma \ref{lem4}, it follows that
\begin{multline*}
\frac{d}{dt}\gamma_3J_3(\xi,t) +\rho_2\gamma_3b_0 (1 -\varepsilon)\xi^2|\hat{y}|^2\leq C(\varepsilon,\gamma_3)\xi^2\int_0^{\infty}g(s)|\hat{\eta}(s)|^2ds+ a\gamma_2\varepsilon\frac{\xi^4}{(1+\xi^2)^2}|\hat{z}|^2 \\
+ C(\varepsilon,\gamma_2,\gamma_3)\xi^2(1+\xi^2)^2\left|\int_0^{\infty}g(s)\hat{\eta}(s)ds\right|^2 +\gamma_3m\xi^4\left|\int_0^{\infty}g(s)\hat{\eta}(s)ds\right|^2+k\gamma_1\tau \varepsilon\frac{\xi^2}{1+\xi^2}|\hat{v}|^2 \\
+C(\varepsilon,\gamma_1,\gamma_3)\xi^2(1+\xi^2)\left|\int_0^{\infty}g(s)\hat{\eta}(s)ds\right|^2+ \gamma_3\delta |\xi|^3|\hat{\theta}|\left|\int_0^{\infty}g(s)\hat{\eta}(s)ds\right|.
\end{multline*}
Thus,
\begin{multline}\label{e33}
\frac{d}{dt}\gamma_3J_3(\xi,t) +\rho_2\gamma_3b_0 (1 -\varepsilon)\xi^2|\hat{y}|^2\leq C(\varepsilon,\gamma_3)\xi^2\int_0^{\infty}g(s)|\hat{\eta}(s)|^2ds+ a\gamma_2\varepsilon\frac{\xi^4}{(1+\xi^2)^2}|\hat{z}|^2 +k\gamma_1\tau\varepsilon\xi^2(1+\xi^2)|\hat{v}|^2\\
+ C(\varepsilon,\gamma_1,\gamma_2,\gamma_3)\xi^2(1+\xi^2)^2\left|\int_0^{\infty}g(s)\hat{\eta}(s)ds\right|^2 + \gamma_3\delta |\xi|^3|\hat{\theta}|\left|\int_0^{\infty}g(s)\hat{\eta}(s)ds\right|.
\end{multline}
Adding \eqref{e32} and \eqref{e33}, we obtain
\begin{multline}\label{e33'}
\frac{d}{dt}\left\lbrace\gamma_1\frac{\xi^2}{1+\xi^2} J_1(\xi,t)+\gamma_2\dfrac{\xi^2}{(1+\xi^2)^2}J_2(\xi,t) + \gamma_3J_3(\xi,t) \right\rbrace + k\left[\gamma_1\tau(1-2\varepsilon)-C(\varepsilon)\gamma_2\right]\frac{\xi^2}{1+\xi^2}|\hat{v}|^2 \\
 +\rho_1(1-2\varepsilon)\gamma_2\dfrac{\xi^4}{(1+\xi^2)^2}|\hat{u}|^2  +a(1-2\varepsilon)\gamma_2\dfrac{\xi^4}{(1+\xi^2)^2}|\hat{z}|^2+\rho_2\left[\gamma_3b_0(1 -\varepsilon)- C(\varepsilon,\gamma_1,\gamma_2)\right]\xi^2|\hat{y}|^2 \\
\leq      C(\varepsilon,\gamma_3)\xi^2\int_0^{\infty}g(s)|\hat{\eta}(s)|^2ds+   C(\varepsilon,\gamma_1,\gamma_2,\gamma_3)\xi^2(1+\xi^2)^2\left|\int_0^{\infty}g(s)\hat{\eta}(s)ds\right|^2 \\ + \gamma_3\delta |\xi|^3|\hat{\theta}|\left|\int_0^{\infty}g(s)\hat{\eta}(s)ds\right|+C(\varepsilon,\gamma_1,\gamma_2,\gamma_3)\xi^2|\hat{q}|^2.
\end{multline}
Now, consider the following functional
\begin{equation*}
\LL_2(\xi,t)= \gamma_1\xi^2 J_1(\xi,t)+\gamma_2\dfrac{\xi^2}{1+\xi^2}J_2(\xi,t)+\gamma_3J_5(\xi,t) +  J_4(\xi,t).
\end{equation*}
From Lemma \ref{lem5}, applying Young inequality to \eqref{e33'} and \eqref{e20}, together with \eqref{ineq}, we have
\begin{multline}\label{e34}
\frac{d}{dt}\left\lbrace\gamma_1\frac{\xi^2}{1+\xi^2} J_1(\xi,t)+\gamma_2\dfrac{\xi^2}{(1+\xi^2)^2}J_2(\xi,t) + \gamma_3J_3(\xi,t) \right\rbrace + k\left[\gamma_1\tau(1-2\varepsilon)-C(\varepsilon)\gamma_2\right]\frac{\xi^2}{1+\xi^2}|\hat{v}|^2 \\
 +\rho_1(1-2\varepsilon)\gamma_2\dfrac{\xi^4}{(1+\xi^2)^2}|\hat{u}|^2  +a(1-2\varepsilon)\gamma_2\dfrac{\xi^4}{(1+\xi^2)^2}|\hat{z}|^2
 +\rho_2\left[\gamma_3b_0(1 -2\varepsilon)- C(\varepsilon,\gamma_1,\gamma_2)\right]\xi^2|\hat{y}|^2 \\ + \rho_3(1-2\varepsilon)\xi^2|\hat{\theta}|^2
\leq     C_1(1+b_0)\xi^2(1+\xi^2)^2\int_0^{\infty}g(s)|\hat{\eta}(s)|^2ds +     C_1(1+\xi^2)|\hat{q}|^2,
\end{multline}
where $C_1$ is a positive constant that depends on $\varepsilon$ and $\gamma_j$ for $j=1,2,3$. In order to make all the coefficients in the right side in \eqref{e34} positive, we have to choose appropriate constant $\gamma_i$. First, let us fix $\varepsilon$ such that $\varepsilon < \frac12.$ Thus, we can take any $\gamma_2 >0$ and
\begin{equation}\label{constant1'}
 \gamma_1 >\frac{C(\varepsilon)\gamma_2}{\tau(1-2\varepsilon)}, \quad \gamma_3 >\frac{C(\varepsilon,\gamma_1,\gamma_2)}{b_0(1-2\varepsilon)}.
\end{equation}
Then, from \eqref{trivial} and \eqref{constant1'}, we can deduce the existence of a positive constant $M_1$ such that
\begin{multline}\label{e35}
\frac{d}{dt}\LL_2(\xi,t) \leq -M_1 \frac{\xi^4}{(1+\xi^2)^2}\left\lbrace k|\hat{v}|^2 +\rho_1|\hat{u}|^2
+a|\hat{z}|^2 +\rho_2|\hat{y}|^2+\rho_3|\hat{\theta}|^2\right\rbrace \\
 + C_1(1+\xi^2)^2\xi^2\int_0^{\infty}g(s)|\hat{\eta}(s)|^2ds  +C_1(1+\xi^2)^2|\hat{q}|^2.
 \end{multline}
Finally, we define the following Lyapunov functional:
\begin{equation*}
\LL(\xi,t) = \LL_2(\xi,t,t)+N(1+\xi^2)^2\hat{E}(\xi,t),
\end{equation*}
where $N$ is a positive constant to be fixed later. Note that the definition of $\LL_2$ together with inequality \eqref{ineq} imply that
\begin{align*}
\left|\LL_2(\xi,t)\right| &\leq M_2 \left\lbrace |J_1(\xi,t)|+|J_2(\xi,t)|+|J_3(\xi,t)|+ |J_4(\xi,t)|\right\rbrace \\
& \leq M_2 (1+\xi^2)^2\hat{E}(\xi,t).
\end{align*}
Hence, we obtain
\begin{equation}\label{e36}
(N-M_2)(1+\xi^2)^2\hat{E}(\xi,t) \leq \LL(\xi,t)\leq (N+M_2)(1+\xi^2)^2 \hat{E}(\xi,t).
\end{equation}
On the other hand, taking the derivative of $\LL$ with respect to $t$ and using the estimates \eqref{e35} and Lemma \ref{lem1}, it follows that
\begin{multline*}
\frac{d}{dt}\LL(\xi,t) \leq -M_1 \frac{\xi^4}{(1+\xi^2)^2}\left\lbrace k|\hat{v}|^2 +\rho_1|\hat{u}|^2
+a|\hat{z}|^2 + \rho_2|\hat{y}|^2+\rho_3|\hat{\theta}|^2\right\rbrace  \\
 -\left(2N\beta-C_1\right)(1+\xi^2)^2|\hat{q}|^2 -\left( Nk_1 - C_1 (1+b_0)\right)(1+\xi^2)^2\xi^2\int_0^{\infty}g(s)|\hat{\eta}(s)|^2ds.
\end{multline*}
Now, choosing $N$ such that $N\geq \max\left\lbrace M_2,\dfrac{C_1}{2\beta}, \dfrac{C_1(1+b_0)}{k_1m}\right\rbrace$ and using the inequality $(1+\xi^2)^2\geq \dfrac{\xi^4}{(1+\xi^2)^2}$, there exists  a positive constant $M_3$ such that
\begin{equation*}
\frac{d}{dt}\LL(\xi,t) \leq - M_3\frac{\xi^4}{(1+\xi^2)^2} \hat{E}(\xi,t).
\end{equation*}
$(\ref{e36})$ implies that
\begin{equation*}
\frac{d}{dt}\LL(\xi,t) \leq - \Gamma\frac{\xi^4}{(1+\xi^2)^4} \LL(\xi,t)
\end{equation*}
where $\Gamma=\dfrac{M_3}{N+M_2}$. By Gronwall's inequality, it follows that
\begin{equation*}
\LL(\xi,t) \leq \LL(\xi,0) e^{-\Gamma \rho(\xi)t}, \qquad \rho(\xi)=\frac{\xi^4}{(1+\xi^2)^4},
\end{equation*}
again by using $(\ref{e36})$, we have that
\begin{align*}
\hat{E}(\xi,t) \leq C \hat{E}(\xi,0) e^{-\Gamma \rho(\xi)t}, \quad \text{where}\quad C= \frac{N+M_2}{N-M_2}> 0.
\end{align*}

\end{proof}


\subsection{The Timoshenko-Fourier Law}
Similarly to the previous case, taking Fourier transform in (\ref{system4}), we obtain the following integro differential system
\begin{align}
&\hat{v}_t-i\xi\hat{u}+\hat{y}=0, \label{eq1}\\
&\rho_1\hat{u}_t-ik\xi\hat{v}=0, \label{eq2}\\
&\hat{z}_t-i\xi\hat{y} =0, \label{eq3}\\
&\rho_2\hat{y}_t -ia\xi\hat{z}+m\xi^2\displaystyle\int_0^{\infty}g(s)\hat{\eta}(s)ds -k\hat{v}+i\delta\xi\hat{\theta}=0,\label{eq4}\\
&\rho_3\hat{\theta}_t+\tilde{\beta}\xi^2\hat{\theta}+i\delta\xi\hat{y}=0,\label{eq5}\\
&\hat{\eta}_t+\hat{\eta}_s-\hat{y}=0, \label{eq6}
\end{align}
where the solution vector and initial data are given by $\hat{V}(\xi,t)=(\hat{v},\hat{u},\hat{z},\hat{y},\hat{\theta},\hat{q},\hat{\eta})^{T}$ and $\hat{V}(\xi,0)=\hat{V}_0(\xi)$, respectively. Furthermore, the energy functional associated to the above system is defined as
\begin{equation}\label{energy2}
\hat{\E}\left (t, \xi\right) = \rho_1|\hat{u}|^{2}+\rho_2|\hat{y}|^{2}+\rho_3 |\hat{\theta}|^{2}+k|\hat{v} |^{2}+a|\hat{z} |^{2} +m\xi^2\int_{0}^{\infty}g(s)|\hat{\eta}(s)|^2ds.
\end{equation}

\begin{lem}\label{lem1'}
The energy of the system \eqref{eq1}-\eqref{eq6}, satisfies
\begin{gather}\label{energyfourier}
\frac{d}{dt}\hat{\E}(\xi,t)\leq -2\tilde{\beta}\xi^2 |\hat{\theta}|^2 -k_1 m \xi^2\int_0^{\infty}g(s)|\hat{\eta}(s)|^2ds,
\end{gather}
where the constant $k_1>0$ is given by $(H_2)$.
\end{lem}
\begin{proof}
Multiplying \eqref{eq1} by $k\overline{\hat{v}}$,  \eqref{eq2} by $\overline{\hat{u}}$, \eqref{eq3}  by $a\overline{\hat{z}}$,  \eqref{eq4}  by $\overline{\hat{y}}$ and \eqref{eq5} by $\overline{\hat{\theta}}$, adding and  taking real part, it follows that
\begin{equation}\label{eq8}
\frac{1}{2}\frac{d}{dt}\left\lbrace \rho_1|\hat{u}|^{2}+\rho_2|\hat{y}|^{2}+\rho_3 |\hat{\theta}|^{2}+k|\hat{v} |^{2}+a|\hat{z} |^{2}\right\rbrace=-\tilde{\beta}\xi^2 |\hat{\theta}|^2 -Re\left(m \xi^2\int_0^{\infty} g(s)\eta(s)\overline{\hat{y}}ds\right).
\end{equation}
On the other hand, taking the conjugate of equation \eqref{eq6}, multiplying the resulting equation by $g(s)\hat{\eta}(t,s,x)$ and integrating with respect to $s$, we obtain
\begin{equation}\label{eq9}
\int_0^{\infty}g(s)\hat{\eta}(s)\overline{\hat{y}}ds=\frac{1}{2}\frac{d}{dt}\int_0^{\infty}g(s)|\hat{\eta}(s)|^2ds +\frac{1}{2}\int_0^{\infty}g(s)\frac{d}{ds}|\hat{\eta}(s)|^2ds.
\end{equation}
Integrating by parts the last term in the left side of \eqref{eq9} and substituting in \eqref{eq8}, it follows that
\begin{equation*}
\frac{d}{dt}\hat{E}(\xi,t)= -2\tilde{\beta}\xi^2 |\hat{\theta}|^2 -m\xi^2\int_0^{\infty}g'(s)|\eta(s)|^2ds.
\end{equation*}
By using $(H_2)$, we obtain \eqref{energyfourier}.
\end{proof}

Since the energy of the system \eqref{eq1}-\eqref{eq6} is dissipative, we expect the exponential decay like in the previous subsection. The principal result of this subsection reads as follows:

\begin{thm}\label{teo2}
Let
\begin{equation*}
\chi_{0}= \left(\rho_2-\dfrac{b\rho_1}{k}\right).
\end{equation*}
Then, for any $t \geq 0$ and $\xi \in \R$, we obtain the following decay rates for the energy of the system \eqref{eq1}-\eqref{eq6}:
\begin{equation*}
\hat{\E}(\xi,t) \leq C e^{-\lambda \rho(\xi)}\hat{\E}(0,\xi),
\end{equation*}
where $C, \lambda$ are positive constants and the function $\rho(\cdot)$ is given by
\begin{equation}\label{eq23''}
\rho(\xi) = \begin{cases}
\dfrac{\xi^4}{(1+\xi^2)^3} & \text{if $\chi_{0}=0$}, \\
\\
\dfrac{\xi^4}{(1+\xi^2)^4}  & \text{if $\chi_{0}\neq 0$}.
\end{cases}
\end{equation}
\end{thm}

Similarly to for Cattaneo's law, we need establish some preliminary results.

\begin{lem}\label{lem2'}
Consider the functional
\begin{equation*}
K_1(\xi,t)=-\rho_2Re(\hat{v}\overline{\hat{y}})-\frac{a\rho_1}{k}Re(\hat{z}\overline{\hat{u}}) +\frac{b_0\rho_1\rho_3}{k\delta}Re(\hat{\theta}\overline{\hat{u}}).
\end{equation*}
Then, for any $\varepsilon>0$, $K_1$ satisfies
\begin{align}\label{eq10'}
\frac{d}{dt}K_1(\xi,t) + k(1-\varepsilon)|\hat{v}|^2 \leq \rho_2 |\hat{y}|^2 + \chi_{0} Re\left(i\xi\overline{\hat{u}}\hat{y}\right)+C(\varepsilon)\xi^2|\hat{\theta}|^2+C(\varepsilon)\xi^4\left|\int_0^{\infty}g(s)\hat{\eta}(s)ds\right|^2+ \frac{b_0\tilde{\beta}\rho_1}{k \delta } \xi^2|\hat{\theta}||\hat{u}|,
\end{align}
for $t\geq0$, where $C(\varepsilon)$ is a positive constant and $\chi_{0}= \left(\rho_2-\dfrac{b\rho_1}{k}\right)$.
\end{lem}
\begin{proof}
Multiplying \eqref{eq1} by $-\rho_2 \overline{\hat{y}}$ and taking real part, we obtain
\begin{align*}
-\rho_2Re\left(\hat{v}_t\overline{\hat{y}}\right) + \rho_2Re\left(i\xi\hat{u}\overline{\hat{y}}\right)-\rho_2|\hat{y}|^2=0.
\end{align*}
Multiplying \eqref{eq4} by $- \overline{\hat{v}}$ and taking real part, it follows that
\begin{align*}
-\rho_2Re\left(\hat{y}_t\overline{\hat{v}}\right) + aRe\left(i\xi\hat{z}\overline{\hat{v}}\right)+k|\hat{v}|^2- mRe\left(\xi^2\overline{\hat{v}}\int_0^{\infty}g(s)\hat{\eta}(s)ds\right)-\delta Re\left(i\xi\hat{\theta}\overline{\hat{v}}\right)=0.
\end{align*}
Adding the above identities,
\begin{multline}\label{eq11}
-\rho_2\frac{d}{dt}Re\left(\hat{v}\overline{\hat{y}}\right)+k|\hat{v}|^2=\rho_2|\hat{y}|^2 - aRe\left(i\xi\hat{z}\overline{\hat{v}}\right)
+ m Re\left(\xi^2\overline{\hat{v}}\int_0^{\infty}g(s)\hat{\eta}(s)ds\right) -\rho_2Re\left(i\xi\hat{u}\overline{\hat{y}}\right)+\delta Re\left(i\xi\hat{\theta}\overline{\hat{v}}\right).
\end{multline}
On the other hand, multiplying \eqref{eq2} by $-\frac{a}{k}\overline{\hat{z}}$, \eqref{eq3} by $-\frac{a\rho_1}{k}\overline{\hat{u}}$, adding the results and taking real part, it follows that
\begin{equation}\label{eq12}
-\frac{a\rho_1}{k}\frac{d}{dt}Re\left(\hat{z}\overline{\hat{u}}\right) = -\frac{a\rho_1}{k}Re\left(i\xi\hat{y}\overline{\hat{u}}\right)-aRe\left(i\xi\hat{v}\overline{\hat{z}}\right).
\end{equation}
Adding \eqref{eq11} and \eqref{eq12}, we have
\begin{multline}\label{eqqq1}
\frac{d}{dt}\left\lbrace -\rho_2Re\left(\hat{v}\overline{\hat{y}}\right) -\frac{a\rho_1}{k}Re\left(\hat{z}\overline{\hat{u}}\right) \right\rbrace +k|\hat{v}|^2 =  \rho_2|\hat{y}|^2 + m Re\left(\xi^2\overline{\hat{v}}\int_0^{\infty}g(s)\hat{\eta}(s)ds\right) \\
+\left(\rho_2-\frac{a\rho_1}{k}\right) Re\left(i\xi\hat{y}\overline{\hat{u}}\right)+\delta Re\left(i\xi\hat{\theta}\overline{\hat{v}}\right).
\end{multline}
Moreover, multiplying \eqref{eq2} by $\frac{\delta}{k}\overline{\hat{\theta}}$ and taking real part,
\begin{equation*}
\frac{\delta\rho_1}{k}Re(\hat{u}_t\overline{\hat{\theta}})-\delta Re(i\xi\hat{v}\overline{\hat{\theta}})=0.
\end{equation*}
Next, multiplying \eqref{eq5} by $\frac{\delta\rho_1}{\rho_3 k}\overline{\hat{u}}$ and taking real part,
\begin{equation*}
\frac{\delta\rho_1}{k}Re(\hat{\theta}_t\overline{\hat{u}})+\frac{\delta\tilde{\beta}\rho_1}{\rho_3 k}Re(\xi^2\hat{\theta}\overline{\hat{u}})+\frac{\delta^2\rho_1}{\rho_3 k}Re(i\xi\hat{y}\overline{\hat{u}})=0.
\end{equation*}
Adding the above identities, we obtain
\begin{equation}\label{eqqq2}
\frac{\delta\rho_1}{k}\frac{d}{dt}Re(\hat{\theta}\overline{\hat{u}})= -\frac{\delta\tilde{\beta}\rho_1}{\rho_3 k}Re(\xi^2\hat{\theta}\overline{\hat{u}})-\frac{\delta^2\rho_1}{\rho_3 k}Re(i\xi\hat{y}\overline{\hat{u}})+\delta Re(i\xi\hat{v}\overline{\hat{\theta}}).
\end{equation}
Computing $\eqref{eqqq1}+\frac{b_0\rho_3}{\delta^2}\eqref{eqqq2}$, it follows that
\begin{multline*}
\frac{d}{dt}\left\lbrace -\rho_2Re\left(\hat{v}\overline{\hat{y}}\right) -\frac{a\rho_1}{k}Re\left(\hat{z}\overline{\hat{u}}\right)+ \frac{b_0\rho_1\rho_3}{k\delta}Re(\hat{\theta}\overline{\hat{u}})\right\rbrace +k|\hat{v}|^2 =  \rho_2|\hat{y}|^2 + m Re\left(\xi^2\overline{\hat{v}}\int_0^{\infty}g(s)\hat{\eta}(s)ds\right) \\
+\left(\rho_2-\frac{a\rho_1}{k}-\frac{b_0\rho_1}{ k}\right) Re\left(i\xi\overline{\hat{u}}\hat{y}\right)+\left(\delta-\frac{b_0\rho_3}{\delta}\right) Re\left(i\xi\hat{\theta}\overline{\hat{v}}\right) -\frac{b_0\tilde{\beta}\rho_1}{k \delta }Re(\xi^2\hat{\theta}\overline{\hat{u}}).
\end{multline*}
Then,
\begin{multline*}
\frac{d}{dt}K_1(\xi,t) +k|\hat{v}|^2 \leq  \rho_2|\hat{y}|^2 + m \xi^2|\hat{v}|\left|\int_0^{\infty}g(s)\hat{\eta}(s)ds\right| +
\chi_0 Re\left(i\xi\overline{\hat{u}}\hat{y}\right)+\left|\delta-\frac{b_0\rho_3}{\delta}\right| |\xi| |\hat{\theta}||\hat{v}| +\frac{b_0\tilde{\beta}\rho_1}{k \delta } \xi^2|\hat{\theta}||\hat{u}|.
\end{multline*}
Applying Young's inequality, \eqref{eq10'} follows.
\end{proof}

\begin{lem}\label{lem3'}
Consider the functional
\begin{equation}\label{eq13}
K_2(\xi,t)=\rho_1Re\left(i\xi\overline{\hat{u}}\hat{v}\right) +\rho_2Re\left(i\xi\hat{y}\overline{\hat{z}}\right).
\end{equation}
For any $\varepsilon >0$, the estimate
\begin{multline}\label{eq14}
\frac{d}{dt}K_2(\xi,t) +\rho_1(1-\varepsilon)\xi^2|\hat{u}|^2+a(1-\varepsilon)\xi^2|\hat{z}|^2\leq C(\varepsilon)(1+\xi^2)|\hat{v}|^2 +C(\varepsilon)(1+\xi^2)|\hat{y}|^2 \\
+C(\varepsilon)\xi^2|\hat{\theta}|^2 +C(\varepsilon)\xi^4\left|\int_0^{\infty}g(s)\hat{\eta}(s)ds\right|^2
\end{multline}
is satisfied.
\end{lem}

\begin{proof}
Multiplying \eqref{eq1} by $i\rho_1\xi \overline{\hat{u}}$ and taking real part,
\begin{align*}
\rho_1Re\left(i\xi\hat{v}_t\overline{\hat{u}}\right) +\rho_1\xi^2|\hat{u}|^2+\rho_1Re\left(i\xi\hat{y}\overline{\hat{u}}\right)=0.
\end{align*}
Multiplying \eqref{eq2} by $-i\xi \overline{\hat{v}}$ and taking real part,
\begin{align*}
-\rho_1Re\left(i\xi\hat{u}_t\overline{\hat{v}}\right) -k\xi^2|\hat{v}|^2=0.
\end{align*}
Adding the above identities, we obtain
\begin{align}\label{eq15}
\rho_1\frac{d}{dt}Re\left(i\xi\hat{v}\overline{\hat{u}}\right)+\rho_1\xi^2|\hat{u}|^2=k\xi^2|\hat{v}|^2 -\rho_1Re\left(i\xi\hat{y}\overline{\hat{u}}\right).
\end{align}
Moreover, multiplying \eqref{eq3} by $-i\rho_2\xi \overline{\hat{y}}$ and taking real part,
\begin{align*}
-\rho_2Re\left(i\xi\hat{z}_t\overline{\hat{y}}\right) -\rho_2\xi^2|\hat{y}|^2=0.
\end{align*}
Multiplying \eqref{eq4} by $i\xi \overline{\hat{z}}$ and taking real part,
\begin{align*}
\rho_2Re\left(i\xi\hat{y}_t\overline{\hat{z}}\right) +a\xi^2|\hat{z}|^2 -kRe\left(i\xi\hat{v}\overline{\hat{z}}\right)+ m Re\left(i\xi^3\overline{\hat{z}}\int_0^{\infty}g(s)\hat{\eta}(s)ds\right) -\delta Re\left(\xi^2\hat{\theta}\overline{\hat{z}}\right)=0.
\end{align*}
Adding the above identities,
\begin{equation}\label{eq16}
\rho_2\frac{d}{dt}Re\left(i\xi\hat{y}\overline{\hat{z}}\right) +a\xi^2|\hat{z}|^2= \rho_2\xi^2|\hat{y}|^2+kRe\left(i\xi\hat{v}\overline{\hat{z}}\right)- m Re\left(i\xi^3\overline{\hat{z}}\int_0^{\infty}g(s)\hat{\eta}(s)ds\right) +\delta Re\left(\xi^2\hat{\theta}\overline{\hat{z}}\right).
\end{equation}
Therefore, computing \eqref{eq15} $+$ \eqref{eq16}, it follows that
\begin{multline*}
\frac{d}{dt}K_2(\xi,t) +\rho_1\xi^2|\hat{u}|^2 +a\xi^2|\hat{z}|^2\leq \rho_2\xi^2|\hat{y}|^2+k\xi^2|\hat{v}|^2+k|\xi||\hat{v}|\hat{z}| \\
- m |\xi|^3|\hat{z}|\left|\int_0^{\infty}g(s)\hat{\eta}(s)ds\right|  +\delta \xi^2|\hat{\theta}||\hat{z}|+\rho_1|\xi||\hat{y}||\hat{u}|.
\end{multline*}
Applying Young's inequality, (\ref{eq14}) holds. \\
\end{proof}

\begin{lem}\label{lem7}
Consider the functional
\begin{equation*}
K_3(\xi,t)=-\rho_2Re\left( \xi^2 \overline{\hat{y}}\int_0^{\infty}g(s)\hat{\eta}(s)ds\right).
\end{equation*}
Then, for any $\varepsilon>0$,  the following estimate
\begin{multline}\label{eq10''}
\frac{d}{dt}K_3(\xi,t) +\rho_2b_0(1-\varepsilon)\xi^2|\hat{y}|^2 \leq C(\varepsilon) \xi^2  \int_0^{\infty}g(s)|\hat{\eta}(s)|^2ds +a |\xi|^3 |\hat{z}|\left| \int_0^{\infty}g(s)\hat{\eta}(s)ds\right| \\
+m\xi^4\left| \int_0^{\infty}g(s)\hat{\eta}(s)ds\right|^2 +k\xi^2 |\hat{v}| \left| \int_0^{\infty}g(s)\hat{\eta}(s)ds\right| +\delta |\xi|^3 |\hat{\theta}| \left| \int_0^{\infty}g(s)\hat{\eta}(s)ds\right|.
\end{multline}
holds.
\end{lem}

\begin{proof}
Proceeding as proof of Lemma \eqref{lem4}, we obtain \eqref{eq10''}. Indeed, we have to multiply \eqref{eq4} by $-\xi^2g(s) \overline{\hat{\eta}}$ and \eqref{eq6} by $-\rho_2\xi^2g(s) \overline{\hat{y}}$, next we take the integration with respect to $s$ for the real parts. We omit the details.
\end{proof}

\vglue 0.3cm

\begin{proof}[\textbf{\large Proof of Theorem \ref{teo2}}]
As Theorem \ref{teo1}, we will consider several cases:
\vglue 0.3 cm

\noindent\textbf{I. Case $\boldsymbol{\chi_0=0}$:} Consider the expressions
\[
\zeta_1\xi^2K_1(\xi,t), \quad \zeta_2\dfrac{\xi^2}{1+\xi^2}K_2(\xi,t), \quad \zeta_3K_3(\xi,t),
\]
where $\zeta_1$, $\zeta_2$ and $\zeta_3$  are positive constants to be fixed later. Lemmas \ref{lem2'} and \ref{lem3'} imply that
\begin{multline*}
\frac{d}{dt}\left\lbrace\zeta_1\xi^2 K_1(\xi,t)+\zeta_2\dfrac{\xi^2}{1+\xi^2}K_2(\xi,t) \right\rbrace + k\left[\zeta_1(1-\varepsilon)-C(\varepsilon)\zeta_2\right]\xi^2|\hat{v}|^2 +\rho_1(1-\varepsilon)\zeta_2\dfrac{\xi^4}{1+\xi^2}|\hat{u}|^2 \\
 +a(1-\varepsilon)\zeta_2\dfrac{\xi^4}{1+\xi^2}|\hat{z}|^2
\leq C(\varepsilon,\zeta_1,\zeta_2)\xi^2|\hat{y}|^2 +C(\varepsilon,\zeta_1,\zeta_2)\xi^4|\hat{\theta}|^2 \\
\frac{b_0\tilde{\beta}\rho_1}{k\delta}\zeta_1\xi^4|\hat{\theta}||\hat{u}|+C(\varepsilon,\zeta_1,\zeta_2)\xi^6\left|\int_0^{\infty}g(s)\hat{\eta}(s)ds\right|^2.
\end{multline*}
By Young's inequality,
\begin{multline}\label{eq24}
\frac{d}{dt}\left\lbrace\zeta_1\xi^2 K_1(\xi,t)+\zeta_2\dfrac{\xi^2}{1+\xi^2}K_2(\xi,t) \right\rbrace + k\left[\zeta_1(1-\varepsilon)-C(\varepsilon)\zeta_2\right]\xi^2|\hat{v}|^2 +\rho_1(1-2\varepsilon)\zeta_2\dfrac{\xi^4}{1+\xi^2}|\hat{u}|^2 \\
 +a(1-\varepsilon)\zeta_2\dfrac{\xi^4}{1+\xi^2}|\hat{z}|^2
\leq C(\varepsilon,\zeta_1,\zeta_2)\xi^2|\hat{y}|^2 +C(\varepsilon,\zeta_1,\zeta_2)\xi^4(1+\xi^2)|\hat{\theta}|^2\\
+C(\varepsilon,\zeta_1,\zeta_2)\xi^6\left|\int_0^{\infty}g(s)\hat{\eta}(s)ds\right|^2.
\end{multline}
Further, applying Young's inequality in \eqref{eq10''} in the Lemma \ref{lem7}, it follows that
\begin{multline*}
\frac{d}{dt}\zeta_3K_3(\xi,t) +\rho_2\zeta_3b_0 (1 -\varepsilon)\xi^2|\hat{y}|^2\leq C(\varepsilon,\zeta_3)\xi^2\int_0^{\infty}g(s)|\hat{\eta}(s)|^2ds+ a\zeta_2\varepsilon\frac{\xi^4}{1+\xi^2}|\hat{z}|^2 \\
+ C(\varepsilon,\zeta_2,\zeta_3)\xi^2(1+\xi^2)\left|\int_0^{\infty}g(s)\hat{\eta}(s)ds\right|^2 +\zeta_3m\xi^4\left|\int_0^{\infty}g(s)\hat{\eta}(s)ds\right|^2+k\zeta_1\varepsilon\xi^2|\hat{v}|^2 \\
+C(\varepsilon,\zeta_1,\zeta_3)\xi^2\left|\int_0^{\infty}g(s)\hat{\eta}(s)ds\right|^2+ \zeta_3\delta |\xi|^3|\hat{\theta}|\left|\int_0^{\infty}g(s)\hat{\eta}(s)ds\right|.
\end{multline*}
Thus,
\begin{multline}\label{eq25}
\frac{d}{dt}\zeta_3K_3(\xi,t) +\rho_2\zeta_3b_0 (1 -\varepsilon)\xi^2|\hat{y}|^2\leq C(\varepsilon,\zeta_3)\xi^2\int_0^{\infty}g(s)|\hat{\eta}(s)|^2ds+ a\zeta_2\varepsilon\frac{\xi^4}{1+\xi^2}|\hat{z}|^2 +k\zeta_1\varepsilon\xi^2|\hat{v}|^2\\
+ C(\varepsilon,\zeta_1,\zeta_2,\zeta_3)\xi^2(1+\xi^2)\left|\int_0^{\infty}g(s)\hat{\eta}(s)ds\right|^2 + C(\varepsilon)\xi^4|\hat{\theta}|^2.
\end{multline}
Computing \eqref{eq24} $+$ \eqref{eq25}, we obtain
\begin{multline*}
\frac{d}{dt}\left\lbrace \zeta_1\xi^2 K_1(\xi,t)+\zeta_2\dfrac{\xi^2}{1+\xi^2}K_2(\xi,t)+\zeta_3K_3(\xi,t) \right\rbrace + k\left[\zeta_1(1-2\varepsilon)-C(\varepsilon)\zeta_2\right]\xi^2|\hat{v}|^2\\
 +\rho_1(1-2\varepsilon)\zeta_2\dfrac{\xi^4}{1+\xi^2}|\hat{u}|^2
+a(1-2\varepsilon)\zeta_2\dfrac{\xi^4}{1+\xi^2}|\hat{z}|^2 +\rho_2\left[\zeta_3b_0(1 -\varepsilon)-C(\varepsilon,\zeta_1,\zeta_2)\right]\xi^2|\hat{y}|^2 \\
\leq C(\varepsilon,\zeta_3)\xi^2\int_0^{\infty}g(s)|\hat{\eta}(s)|^2ds +C(\varepsilon,\zeta_1,\zeta_2,\zeta_3)\xi^2(1+\xi^2)^2\left|\int_0^{\infty}g(s)\hat{\eta}(s)ds\right|^2 \\
 + C(\varepsilon,\zeta_1,\zeta_2,\zeta_3)\xi^2(1+\xi^2)^2|\hat{\theta}|^2.
\end{multline*}
From inequality \eqref{ineq}, we conclude that
\begin{multline}\label{eq26}
\frac{d}{dt}\left\lbrace \zeta_1\xi^2 K_1(\xi,t)+\zeta_2\dfrac{\xi^2}{1+\xi^2}K_2(\xi,t)+\zeta_3K_3(\xi,t) \right\rbrace + k\left[\zeta_1(1-2\varepsilon)-C(\varepsilon)\zeta_2\right]\xi^2|\hat{v}|^2\\
 +\rho_1(1-2\varepsilon)\zeta_2\dfrac{\xi^4}{1+\xi^2}|\hat{u}|^2
+a(1-2\varepsilon)\zeta_2\dfrac{\xi^4}{1+\xi^2}|\hat{z}|^2 +\rho_2\left[\zeta_3b_0(1 -\varepsilon)-C(\varepsilon,\zeta_1,\zeta_2)\right]\xi^2|\hat{y}|^2 \\
\leq C(\varepsilon,\zeta_1,\zeta_2,\zeta_3)(1+b_0)\xi^2(1+\xi^2)^2 \int_0^{\infty}g(s)|\hat{\eta}(s)|^2ds + C(\varepsilon,\zeta_1,\zeta_2,\zeta_3)\xi^2(1+\xi^2)^2|\hat{\theta}|^2.
\end{multline}

In order to make all coefficients in the right-hand side in \eqref{eq26} positive, we have to choose appropriate constant $\zeta_i$. First, let us fix $\varepsilon$, such that $\varepsilon< \frac12.$ Thus, we can take any $\zeta_2 >0$ and
\begin{equation}\label{constant2}
\zeta_1 >\frac{C(\varepsilon)\zeta_2}{1-2\varepsilon}, \quad \zeta_3 >\frac{C(\varepsilon,\zeta_1,\zeta_2)}{b_0(1-\varepsilon)}.
\end{equation}
Then, from \eqref{constant2} and the estimate $\frac{\xi^2}{1+\xi^2}\leq 1$, we can deduce the existence of a positive constant $M_1$ such that
\begin{multline}\label{eq28}
\frac{d}{dt}\Q_1(\xi,t) \leq -M_1 \frac{\xi^4}{1+\xi^2}\left\lbrace k|\hat{v}|^2 +\rho_1|\hat{u}|^2
+a|\hat{z}|^2 +\rho_2|\hat{y}|^2\right\rbrace \\ +C_1(1+b_0)(1+\xi^2)^2\xi^2\int_0^{\infty}g(s)|\hat{\eta}(s)|^2ds
 +C_1(1+\xi^2)^2\xi^2|\hat{\theta}|^2,
\end{multline}
where
\begin{equation*}
\Q_1(\xi,t)= \zeta_1\xi^2 K_1(\xi,t)+\zeta_2\dfrac{\xi^2}{1+\xi^2}K_2(\xi,t)+\zeta_3K_3(\xi,t)
\end{equation*}
and  $C_1$ is a positive constant that depends of $\varepsilon$ and $\zeta_j$ for $j=1,2,3$. Finally, we define the following Lyapunov functional:
\begin{equation*}
\Q(\xi,t) = \Q_1(\xi,t,t)+N(1+\xi^2)^2\hat{\E}(\xi,t),
\end{equation*}
where $N$ is a positive constant to be fixed later. Note that the definition of $\Q_1$ together with \eqref{ineq} imply that
\begin{align*}
\left|\Q_1(\xi,t)\right| &\leq M_2 \left\lbrace \xi^2|K_1(\xi,t)|+|K_2(\xi,t)|+|K_3(\xi,t)|\right\rbrace \\
& \leq M_2 (1+\xi^2)\hat{\E}(\xi,t).
\end{align*}
Hence, we obtain
\begin{equation}\label{eq29+}
(N-M_2)(1+\xi^2)^2\hat{\E}(\xi,t) \leq \Q(\xi,t)\leq (N+M_2)(1+\xi^2)^2 \hat{\E}(\xi,t).
\end{equation}
On the other hand, taking the derivative of $\Q$ with respect to $t$ and using the estimates \eqref{eq28} and Lemma \ref{lem1'}, it follows that
\begin{multline*}
\frac{d}{dt}\Q(\xi,t) \leq -M_1 \frac{\xi^4}{1+\xi^2}\left\lbrace k|\hat{v}|^2 +\rho_1|\hat{u}|^2
+a|\hat{z}|^2 + \rho_2|\hat{y}|^2\right\rbrace   -\left(2N\tilde{\beta}-C_1\right)(1+\xi^2)^2\xi^2|\hat{\theta}|^2 \\
-\left( Nk_1m - C_1(1+ b_0)\right)(1+\xi^2)^2\xi^2\int_0^{\infty}g(s)|\hat{\eta}(s)|^2ds.
\end{multline*}
Now, choosing $N$ such that $N\geq \max\left\lbrace M_2,\dfrac{C_1}{2\tilde{\beta}}, \dfrac{C_1(1+b_0)}{k_1 m}\right\rbrace$ and using the inequalities $(1+\xi^2)^2 \geq \dfrac{\xi^4}{1+\xi^2}$ and $(1+\xi^2)^2 \geq \dfrac{\xi^2}{1+\xi^2}$, there exists  a positive constant $M_3$ such that
\begin{equation*}
\frac{d}{dt}\Q(\xi,t) \leq - M_3\frac{\xi^4}{1+\xi^2} \hat{\E}(\xi,t).
\end{equation*}
Estimate $(\ref{eq29+})$ implies that
\begin{equation*}
\frac{d}{dt}\Q(\xi,t) \leq - \Gamma\frac{\xi^4}{(1+\xi^2)^3} \Q(\xi,t)
\end{equation*}
where $\Gamma=\dfrac{M_3}{(N+M_2)}$. By Gronwall's inequality, it follows that
\begin{equation*}
\Q(\xi,t) \leq  e^{-\Gamma \rho(\xi)t}\Q(\xi,0), \quad \rho(\xi)=\frac{\xi^4}{(1+\xi^2)^3}.
\end{equation*}
Again by using $(\ref{eq29+})$, we have that
\begin{align*}
\hat{\E}(\xi,t) \leq C  e^{-\Gamma \rho(\xi)t}\hat{\E}(\xi,0), \quad \text{where}\quad C= \frac{N+M_2}{N-M_2}> 0.
\end{align*}
\vglue 0.3 cm
\noindent\textbf{II. Case $\boldsymbol{\chi_0 \neq 0}$}: Similar to previous case, we introduce positive constants $\kappa_1$, $\kappa_2$ and $\kappa_3$ that will be fixed later. Next, we estimate the following term by applying Young's inequality,
\begin{align*}
\begin{cases}
\left| \chi_0 Re(i\xi\hat{u}\overline{\hat{y}})\right| \leq \dfrac{\rho_1 \kappa_2\varepsilon}{2\kappa_1} \dfrac{\xi^2}{1+\xi^2}|\hat{u}|^2+C(\varepsilon,\kappa_1,\kappa_2)(1+\xi^2)|\hat{y}|^2, \\
\\
\dfrac{b_0\beta\rho_1}{k \delta } \xi^2|\hat{\theta}||\hat{u}|\leq \dfrac{\rho_1 \kappa_2\varepsilon}{2\kappa_1} \dfrac{\xi^2}{1+\xi^2}|\hat{u}|^2+C(\varepsilon,\kappa_1,\kappa_2)(1+\xi^2)\xi^2|\hat{\theta}|^2,
\end{cases}
\end{align*}
Hence, \eqref{eq10'} can be written as
\begin{multline}\label{eq31}
\frac{d}{dt}K_1(\xi,t) + k(1-\varepsilon)|\hat{v}|^2 \leq  \frac{\rho_1 \kappa_2\varepsilon}{\kappa_1} \frac{\xi^2}{1+\xi^2}|\hat{u}|^2+C(\varepsilon,\kappa_1,\kappa_2)(1+\xi^2)|\hat{y}|^2  \\
+C(\varepsilon, \kappa_1,\kappa_2)(1+\xi^2)\xi^2|\hat{\theta}|^2+C(\varepsilon)\xi^4\left|\int_0^{\infty}g(s)\hat{\eta}(s)ds\right|^2.
\end{multline}
Thus, inequalities \eqref{eq14} and \eqref{eq31} imply that
\begin{multline}\label{eq32}
\frac{d}{dt}\left\lbrace\kappa _1\frac{\xi^2}{1+\xi^2} K_1(\xi,t)+\kappa _2\dfrac{\xi^2}{(1+\xi^2)^2}K_2(\xi,t) \right\rbrace + k\left[\kappa _1(1-\varepsilon)-C(\varepsilon)\kappa _2\right]\frac{\xi^2}{1+\xi^2}|\hat{v}|^2 \\
+\rho_1(1-2\varepsilon)\kappa _2\dfrac{\xi^4}{(1+\xi^2)^2}|\hat{u}|^2
 +a(1-\varepsilon)\kappa _2\dfrac{\xi^4}{(1+\xi^2)^2}|\hat{z}|^2 \\
\leq C(\varepsilon,\kappa _1,\kappa _2)\xi^2|\hat{y}|^2
 +C(\varepsilon,\kappa _1,\kappa _2)\xi^4|\hat{\theta}|^2
+C(\varepsilon, \kappa _1,\kappa _2)\xi^4\left|\int_0^{\infty}g(s)\hat{\eta}(s)ds\right|^2
\end{multline}
Furthermore, applying Young's inequality in \eqref{eq10''} in the Lemma \ref{lem7}, it follows that
\begin{multline*}
\frac{d}{dt}\kappa_3K_3(\xi,t) +\rho_2\kappa_3b_0 (1 -\varepsilon)\xi^2|\hat{y}|^2\leq C(\varepsilon,\kappa_3)\xi^2\int_0^{\infty}g(s)|\hat{\eta}(s)|^2ds+ a\kappa_2\varepsilon\frac{\xi^4}{(1+\xi^2)^2}|\hat{z}|^2 \\
+ C(\varepsilon,\kappa_2,\kappa_3)\xi^2(1+\xi^2)^2\left|\int_0^{\infty}g(s)\hat{\eta}(s)ds\right|^2 +\kappa_3m\xi^4\left|\int_0^{\infty}g(s)\hat{\eta}(s)ds\right|^2+k\kappa_1\varepsilon\frac{\xi^2}{1+\xi^2}|\hat{v}|^2 \\
+C(\varepsilon,\kappa_1,\kappa_3)\xi^2(1+\xi^2)\left|\int_0^{\infty}g(s)\hat{\eta}(s)ds\right|^2+ \kappa_3\delta |\xi|^3|\hat{\theta}|\left|\int_0^{\infty}g(s)\hat{\eta}(s)ds\right|.
\end{multline*}
Thus,
\begin{multline}\label{eq33}
\frac{d}{dt}\kappa_3K_3(\xi,t) +\rho_2\kappa_3b_0 (1 -\varepsilon)\xi^2|\hat{y}|^2\leq C(\varepsilon,\kappa_3)\xi^2\int_0^{\infty}g(s)|\hat{\eta}(s)|^2ds+ a\kappa_2\varepsilon\frac{\xi^4}{(1+\xi^2)^2}|\hat{z}|^2 +k\kappa_1\varepsilon\frac{\xi^2}{1+\xi^2}|\hat{v}|^2\\
+ C(\varepsilon,\kappa_1,\kappa_2,\kappa_3)\xi^2(1+\xi^2)^2\left|\int_0^{\infty}g(s)\hat{\eta}(s)ds\right|^2 + C(\varepsilon)\xi^4|\hat{\theta}|^2
\end{multline}
Computing \eqref{eq32} $+$ \eqref{eq33}, we obtain
\begin{multline*}
\frac{d}{dt}\left\lbrace\kappa _1\frac{\xi^2}{1+\xi^2} K_1(\xi,t)+\kappa _2\dfrac{\xi^2}{(1+\xi^2)^2}K_2(\xi,t) + \kappa _3K_3(\xi,t) \right\rbrace + k\left[\kappa _1(1-2\varepsilon)-C(\varepsilon)\kappa _2\right]\frac{\xi^2}{1+\xi^2}|\hat{v}|^2 \\
 +\rho_1(1-2\varepsilon)\kappa _2\dfrac{\xi^4}{(1+\xi^2)^2}|\hat{u}|^2  +a(1-2\varepsilon)\kappa _2\dfrac{\xi^4}{(1+\xi^2)^2}|\hat{z}|^2+\rho_2\left[\kappa _3b_0(1 -\varepsilon)- C(\varepsilon,\kappa _1,\kappa _2)\right]\xi^2|\hat{y}|^2 \\
\leq C(\varepsilon,\kappa_3)\xi^2\int_0^{\infty}g(s)|\hat{\eta}(s)|^2ds+ C(\varepsilon,\kappa _1,\kappa _2,\kappa _3)\xi^4|\hat{\theta}|^2
+ C(\varepsilon, \kappa _1,\kappa _2,\kappa _3)\xi^2(1+\xi^2)^2\left|\int_0^{\infty}g(s)\hat{\eta}(s)ds\right|^2.
\end{multline*}
From inequality \eqref{ineq}, we conclude that
\begin{multline}\label{eq34'}
\frac{d}{dt}\left\lbrace\kappa _1\frac{\xi^2}{1+\xi^2} K_1(\xi,t)+\kappa _2\dfrac{\xi^2}{(1+\xi^2)^2}K_2(\xi,t) + \kappa _3K_3(\xi,t) \right\rbrace + k\left[\kappa _1(1-2\varepsilon)-C(\varepsilon)\kappa _2\right]\frac{\xi^2}{1+\xi^2}|\hat{v}|^2 \\
 +\rho_1(1-2\varepsilon)\kappa _2\dfrac{\xi^4}{(1+\xi^2)^2}|\hat{u}|^2  +a(1-2\varepsilon)\kappa _2\dfrac{\xi^4}{(1+\xi^2)^2}|\hat{z}|^2+\rho_2\left[\kappa _3b_0(1 -\varepsilon)- C(\varepsilon,\kappa _1,\kappa _2)\right]\xi^2|\hat{y}|^2 \\
\leq  C(\varepsilon, \kappa _1,\kappa _2,\kappa _3)(1+b_0)\xi^2(1+\xi^2)^2\int_0^{\infty}g(s)|\hat{\eta}(s)|^2ds+ C(\varepsilon,\kappa _1,\kappa _2,\kappa _3)\xi^2(1+\xi^2)^2|\hat{\theta}|^2.
\end{multline}
In order to make all coefficients in the right-hand side in \eqref{eq34'} positive, we have to choose appropriate constant $\kappa _i$. First, let us fix $\varepsilon$, such that $\varepsilon< \frac12.$ Thus, we can take $\kappa_2>0$ and
\begin{equation}\label{constant3}
\kappa _1 >\frac{C(\varepsilon)\kappa _2}{1-2\varepsilon}, \quad \kappa _3 >\frac{C(\varepsilon,\kappa _1,\kappa _2)}{b_0(1-\varepsilon)}.
\end{equation}
Then, from \eqref{constant3} and the estimate $\frac{\xi^2}{1+\xi^2}\leq 1$, we can deduce the existence of a positive constant $M_1$ such that
\begin{multline}\label{eq35}
\frac{d}{dt}\Q_1(\xi,t) \leq -M_1 \frac{\xi^4}{(1+\xi^2)^2}\left\lbrace k|\hat{v}|^2 +\rho_1|\hat{u}|^2
+a|\hat{z}|^2 +\rho_2|\hat{y}|^2\right\rbrace \\
+ C_1(1+b_0)(1+\xi^2)^2\xi^2\int_0^{\infty}g(s)|\hat{\eta}(s)|^2ds+C_1(1+\xi^2)^2\xi^2|\hat{\theta}|^2,
\end{multline}
where
\begin{equation*}
\Q_1(\xi,t)= \kappa _1\frac{\xi^2}{1+\xi^2} K_1(\xi,t)+\kappa _2\dfrac{\xi^2}{(1+\xi^2)^2}K_2(\xi,t)+\kappa _3K_3(\xi,t)
\end{equation*}
and $C_1$ is a positive constant that depends of $\varepsilon$ and $\kappa _j$ for $j=1,2,3$. Finally, we define the following Lyapunov functional:
\begin{equation*}
\Q(\xi,t) = \Q_1(\xi,t,t)+N(1+\xi^2)^2\hat{\E}(\xi,t),
\end{equation*}
where $N$ is a positive constant to be fixed later. Note that the definition of $\Q_1$ together with the inequality \eqref{ineq} imply that
\begin{align*}
\left|\Q_1(\xi,t)\right| & \leq M_2 (1+\xi^2)^2\hat{\E}(\xi,t)
\end{align*}
for some positive constant $M_2$. Hence, we obtain
\begin{equation}\label{eq36}
(N-M_2)(1+\xi^2)^2\hat{\E}(\xi,t) \leq \Q(\xi,t)\leq (N+M_2)(1+\xi^2)^2 \hat{\E}(\xi,t).
\end{equation}
On the other hand, taking the derivative of $\Q$ with respect to $t$ and using the estimates \eqref{eq35} and Lemma \ref{lem1'}, it follows that
\begin{multline*}
\frac{d}{dt}\Q(\xi,t) \leq -M_1 \frac{\xi^4}{(1+\xi^2)^2}\left\lbrace k|\hat{v}|^2 +\rho_1|\hat{u}|^2
+a|\hat{z}|^2 + \rho_2|\hat{y}|^2\right\rbrace   -\left(2N\tilde{\beta}-C_1\right)(1+\xi^2)^2\xi^2|\hat{\theta}|^2 \\
 -\left( Nk_1m - C_1 (1+b_0)\right)(1+\xi^2)^2\xi^2\int_0^{\infty}g(s)|\hat{\eta}(s)|^2ds.
\end{multline*}
Now, choosing $N$ such that $N\geq \max\left\lbrace M_2,\dfrac{C_1}{2\tilde{\beta}}, \dfrac{C_1(1+b_0)}{k_1 m}\right\rbrace$ and noting that  $(1+\xi^2)^2 \geq \dfrac{\xi^4}{(1+\xi^2)^2}$ and $(1+\xi^2)^2 \geq \dfrac{\xi^2}{(1+\xi^2)^2}$, there exists  a positive constant $M_3$ such that
\begin{equation*}
\frac{d}{dt}\Q(\xi,t) \leq - M_3\frac{\xi^4}{(1+\xi^2)^2} \hat{\E}(\xi,t).
\end{equation*}
Estimate $(\ref{eq36})$ implies that
\begin{equation*}
\frac{d}{dt}\Q(\xi,t) \leq - \Gamma \frac{\xi^4}{(1+\xi^2)^4} \Q(\xi,t)
\end{equation*}
where $\Gamma =\dfrac{M_3}{(N+M_2)}$. By Gronwall's inequality, it follows that
\begin{equation*}
\Q(\xi,t) \leq  e^{-\Gamma  \rho(\xi)t}\Q(\xi,0), \quad \rho(\xi)=\frac{\xi^4}{(1+\xi^2)^4},
\end{equation*}
again by using $(\ref{e29})$, we have that
\begin{align*}
\hat{\E}(\xi,t) \leq C  e^{-\Gamma  \rho(\xi)t}\hat{\E}(\xi,0), \quad \text{where}\quad C= \frac{N+M_2}{N-M_2}> 0.
\end{align*}
\end{proof}

\section{The decay estimates}\label{section4}

In this section, we establish the decay rates of solutions $U(x, t)$, $V(x, t)$ of systems \eqref{system2}-\eqref{system20} and \eqref{system4}-\eqref{system40}, respectively. By using the energy inequalities in the Fourier space, we show that the decay rates depend of condition $\chi_{0,\tau}=0$ or $\chi_{0,\tau}\neq 0$ (resp.  $\chi_{0}=0$ or $\chi_{0}\neq 0$). In any case, the regularity loss phenomenon is present. This first main result reads as follows:
\begin{thm}\label{teocattaneo}
Let $s$ be a nonnegative integer and
\begin{equation}\label{propagation}
\chi_{0,\tau}=\left( \tau-\frac{\rho_1}{\rho_3 k}\right)\left( \rho_2 -\frac{b \rho_1}{k}\right) - \frac{\tau \rho_1 \delta^2}{\rho_3 k}.
\end{equation}
Suppose that $U_0 \in \mathbb{H}^s(\R)\cap \mathbb{L}^1(\R)$, where
$$
\mathbb{H}^s(\R):=\left[H^s(\R)\right]^6\times L^2_g(\R;H^{1+s}(\R))\qquad and \qquad \mathbb{L}^1(\R):=\left[L^1(\R)\right]^6\times L^2_g(\R;W^{1,1}(\R)).
$$
Then, the solution $U$ of the system \eqref{system2}, satisfies the following decay estimates,
\begin{enumerate}
\item[(i)] If $\chi_{0,\tau} =0$, then
\begin{align}\label{e38}
\|\partial^k_xU(t)\|_2\leq C(1+t)^{-\frac{1}{8}-\frac{k}{4}}\|U_0\|_1 + C (1+t)^{-\frac{l}{2}}\|\partial_x^{k+l}U_0\|_2, \quad t\geq 0.
\end{align}
\item[(ii)] If $\chi_{0,\tau}\neq 0$, then
\begin{align}\label{e39}
\|\partial^k_xU(t)\|_2\leq C(1+t)^{-\frac{1}{8}-\frac{k}{4}}\|U_0\|_1 + C(1+t)^{-\frac{l}{4}}\|\partial_x^{k+l}U_0\|_2,\quad t\geq 0.
\end{align}
\end{enumerate}
where $k+l \leq s$, $C$ and $c$ are two positive constants.
\end{thm}

\begin{proof}
Applying the Plancherel's identity, we have
\begin{align*}
\|\partial^k_xU(t)\|_2^2 = \|(i\xi)^k\hat{U}(t)\|_2^2 = \int_{\R}|\xi|^{2k}\left|\hat{U}(\xi,t)\right|^2 d\xi.
\end{align*}
It is easy to see that
\begin{align}\label{e40}
c_1\left|\hat{U}(\xi,t)\right|^2 \leq \hat{E}(\xi,t) \leq c_2\left|\hat{U}(\xi,t)\right|^2,
\end{align}
for some positive constant $c_1$ and $c_2$. Thus, it follows that
\begin{align*}
\|\partial^k_xU(t)\|_2^2 \leq \frac{1}{c_1} \int_{\R}|\xi|^{2k}\hat{E}(\xi,t) d\xi.
\end{align*}
From Theorems \ref{teo1} and \eqref{e40}, there exist a postie constant $M>0$, such that
\begin{align*}
\|\partial^k_xU(t)\|_2^2 &\leq M\int_{\R}|\xi|^{2k}e^{-\lambda \rho(\xi)t}\left|\hat{U}(0,\xi)\right|^2 d\xi \\
&\leq M \underbrace{\int_{\left\lbrace|\xi|\leq 1\right\rbrace}|\xi|^{2k}e^{-\lambda \rho(\xi)t}\left|\hat{U}_0(\xi)\right|^2 d\xi}_{I_1} + \underbrace{ M\int_{\left\lbrace|\xi|\geq 1\right\rbrace}|\xi|^{2k}e^{-\lambda \rho(\xi)t}\left|\hat{U}_0(\xi)\right|^2 d\xi}_{I_2}
\end{align*}

\vglue 0.3 cm
\noindent\textbf{Case $\boldsymbol{\chi_{0,\tau} = 0}$:} It is not difficult to see that the function $\rho(\cdot)$  satisfies
\begin{equation}\label{e41}
\left\lbrace\begin{tabular}{l c l}
$\rho(\xi) \geq \frac{1}{8}\xi^4$ & if & $|\xi|\leq 1$, \\
\\
$\rho(\xi) \geq \frac{1}{8}\xi^{-2}$ & if & $|\xi|\geq 1$.
\end{tabular}\right.
\end{equation}
Thus, we estimate $I_1$ as follows,
\begin{equation*}
I_1 \leq M \|\hat{U_0}\|_{L^\infty}^2 \int_{|\xi|\leq 1}|\xi|^{2k}e^{-\frac{\lambda}{8}\xi^4t} d\xi \leq C_1\|\hat{U_0}\|_{L^\infty}^2\left( 1+t\right)^{-\frac{1}{4}(1+2k)} \leq C_1\left( 1+t\right)^{-\frac{1}{4}(1+2k)} \|U_0\|_{L^1}^2.
\end{equation*}
On the other hand, by using the second inequality in (\ref{e41}), we obtain
\begin{align*}
I_2 &\leq M\int_{|\xi|\geq 1}|\xi|^{2k}e^{-\frac{\lambda}{8}\xi^{-2}t}\left|\hat{U_0}(\xi)\right|^2d\xi \leq M \sup_{|\xi| \geq 1}\{ |\xi |^{-2l}e^{-\frac{\lambda}{8} \xi ^{-2}t} \} \int_{\R}|\xi|^{2(k+l)}\left|\hat{U_0}^2(\xi)\right|^2d\xi  \notag \\
&\leq C_2 (1+t)^{-l} \|\partial_x^{k+l}U_0\|_2^2.
\end{align*}
Combining the estimates of $I_1$ and $I_2$, we obtain $(\ref{e38})$.

\vglue 0.3 cm
\noindent\textbf{Case $\boldsymbol{\chi_{0,\tau} \neq 0}$:} In this cases, the function $\rho(\cdot)$ satisfies
\begin{equation}\label{e42}
\left\lbrace\begin{tabular}{l c l}
$\rho(\xi) \geq \frac{1}{16}\xi^4$ & if & $|\xi|\leq 1$ \\
\\
$\rho(\xi) \geq \frac{1}{16}\xi^{-4}$ & if & $|\xi|\geq 1$
\end{tabular}\right.
\end{equation}
Thus, we estimate $I_1$ as following,
\begin{equation*}
I_1 \leq M \|\hat{U_0}\|_{L^\infty}^2 \int_{|\xi|\leq 1}|\xi|^{2k}e^{-\frac{\lambda}{16}\xi^4t} d\xi \leq C_1\|\hat{V_0}\|_{L^\infty}^2\left( 1+t\right)^{-\frac{1}{4}(1+2k)} \leq C_1\left( 1+t\right)^{-\frac{1}{4}(1+2k)} \|U_0\|_{L^1}^2.
\end{equation*}
Moreover, by using the second inequality in (\ref{e42}), it follows that
\begin{align*}
I_2 &\leq M\int_{|\xi|\geq 1}|\xi|^{2k}e^{-\frac{\lambda}{16}\xi^{-4}t}\left|\hat{U_0}(\xi)\right|^2d\xi \leq M \sup_{|\xi| \geq 1}\{ |\xi |^{-2l}e^{-\frac{\lambda}{16} \xi ^{-4}t} \} \int_{\R}|\xi|^{2(k+l)}\left|\hat{U_0}^2(\xi)\right|^2d\xi  \notag \\
&\leq C_2 (1+t)^{-\frac{l}{2}} \|\partial_x^{k+l}U_0\|_2^2.
\end{align*}
Combining the estimates of $I_1$ and $I_2$, we obtain $(\ref{e39})$.
\end{proof}
Similar to the proof of Theorem \ref{teocattaneo},  we establish decay estimates of the solution $V(x,t)$ of Timoshenko-Fourier system \eqref{system4}-\eqref{system40}. The proof of next theorem is carried out by the same technique as that of Theorem \ref{teocattaneo}. Therefore, we omit it.
\begin{thm}\label{teofourier}
Let $s$ be a nonnegative integer and
\begin{equation}\label{propagation2}
\chi_{0}=\left( \rho_2 -\frac{b \rho_1}{k}\right).
\end{equation}
Suppose that $V_0 \in \mathbb{H}^s(\R)\cap \mathbb{L}^1(\R)$, where
$$
\mathbb{H}^s(\R):=\left[H^s(\R)\right]^5\times L^2_g(\R;H^{1+s}(\R))\qquad and \qquad \mathbb{L}^1(\R):=\left[L^1(\R)\right]^5\times L^2_g(\R;W^{1,1}(\R)).
$$
Then, the solution $V$ of the system \eqref{system4}, satisfies the following decay estimates,
\begin{enumerate}
\item[(i)] If $\chi_{0} =0$, then
\begin{align}\label{eq38}
\|\partial^k_xV(t)\|_2\leq C_1(1+t)^{-\frac{1}{8}-\frac{k}{4}}\|V_0\|_1 + C_2(1+t)^{-\frac{l}{2}}\|\partial_x^{k+l}V_0\|_2, \quad t\geq 0.
\end{align}
\item[(ii)] If $\chi_{0}\neq 0$, then
\begin{align}\label{eq38'}
\|\partial^k_xV(t)\|_2\leq C_1(1+t)^{-\frac{1}{8}-\frac{k}{4}}\|V_0\|_1 + C_2(1+t)^{-\frac{l}{4}}\|\partial_x^{k+l}V_0\|_2, \quad t\geq 0.
\end{align}
\end{enumerate}
where $k+l \leq s$, $C_1,C_2$ are two positive constants.
\end{thm}

\section*{Acknowledgments}
The first author was partially supported by Facultad de Ciencias Exactas y Naturales, Unversidad Nacional de Colombia Sede Manizales, under project number 45511.





\end{document}